\documentclass[11pt,oneside,leqno]{amsart}
\usepackage{amsxtra}
\usepackage{amsopn}
\usepackage{amsmath,amsthm,amssymb}
\usepackage{amscd}
\usepackage{amsfonts}
\usepackage{latexsym}
\usepackage{verbatim}
\usepackage[matrix,arrow]{xy}
\usepackage{array}
\usepackage{multirow}
\usepackage{bigstrut}
\usepackage{rotating}
\usepackage{color,enumitem}

\theoremstyle{plain}

\newtheorem{theorem}{Theorem}[section]
\newtheorem*{theorem*}{Theorem}
\newtheorem{definition}[theorem]{Definition}

\newtheorem{prop}[theorem]{Proposition}
\newtheorem{cor}[theorem]{Corollary}

\newtheorem*{mt*}{Main Theorem}
\theoremstyle{remark}
\newtheorem{remark}{Remark}

\newcommand\Q{{\mathbb Q}}
\newcommand\RR{{\mathbb R}}
\newcommand\ZZ{{\mathbb Z}}

\newcommand\T{{\mathbb T}}

\newcommand{\SSS}{\mathbb{S}}

\newcommand\ad{{\mathrm {ad}}}
\newcommand{\Ad}{{\mathrm {Ad\,}}}

\newcommand{\cg}{{\mathfrak{g}}}

\newcommand{\ch}{{\mathfrak{h}}}

\newcommand{\cn}{{\mathfrak{n}}}

\newcommand{\beq}{\begin{equation}}
\newcommand{\eeq}{\end{equation}}
\newcommand{\bea}{\begin{eqnarray}}
\newcommand{\eea}{\end{eqnarray}}

\newcommand{\MinPol}{{\mathrm {MinPol\,}}}

\newenvironment{description*}%
  {\begin{description}%
    \setlength{\itemsep}{3pt}%
    \setlength{\parskip}{0pt}}%
  {\end{description}}

\begin{document}

\title[Lattices, cohomology and models of 6 dim almost abelian solvmanifolds]{Lattices, cohomology and models of six dimensional\\ almost abelian solvmanifolds}
\author{Sergio Console}
\date{\today}
\address{S. Console: Dipartimento di Matematica G. Peano \\ Universit\`a di Torino\\
Via Carlo Alberto 10\\
10123 Torino\\ Italy.} \email{sergio.console@unito.it}

\author{Maura Macr\`i}
\address{M. Macr\`i: Dipartimento di Matematica G. Peano \\ Universit\`a di Torino\\
Via Carlo Alberto 10\\
10123 Torino\\ Italy.} \email{maura.macri@unito.it}

\thanks{This work was supported by  MIUR and by GNSAGA of INdAM
}

\begin{abstract}
We construct lattices on six dimensional not completely solvable almost abelian Lie groups, for which the Mostow condition does not hold. For the corresponding compact quotients, we compute the de Rham cohomology (which does not agree in general with the Lie algebra one) and a minimal model. We show that some of these solvmanifolds admit not invariant symplectic structures and we study formality and Lefschetz properties.
\end{abstract}
\maketitle

\section{Introduction}

A solvmanifold $M$ 
is a compact homogeneous space $M=G/\Gamma$, where $G$ is  a connected and simply connected solvable Lie group and $\Gamma$ is  a lattice in $G$ (that is, a discrete subgroup with compact quotient space).

\smallskip

In the special case of nilmanifolds (i.e., the particular case the solvable Lie group is nilpotent), if the structure constants are rational, a lattice can be always found \cite{Mal}, while for solvmanifolds its existence is harder to establish. 

 \medskip

Lattices determine the topology of  solvmanifolds and are actually their fundamental groups (indeed solvmanifolds are Eilenberg-MacLane spaces of type $K(\pi ,1)$, i.e. all their homotopy groups vanish, besides the first).
Actually, lattices of solvmanifolds yield their diffeomorphism class (cf. Theorem~\ref{diffeo}).

\smallskip

Much of the rich structure of solvmanifolds is encoded by the Mostow fibration (see Section~\ref{Mostow})
$ N / \Gamma_N = (N \Gamma) / \Gamma \hookrightarrow G / \Gamma
\longrightarrow G / (N \Gamma) = \T^k\,$, 
where $\T^k$ is a ($k$-dimensional) torus and $N$ is the nilradical of $G$ (the largest nilpotent normal subgroup of $G$). In general, the Mostow bundle is not principal.

\smallskip

An important special case is provided if  the Lie algebra ${\mathfrak g}$ of $G$  has an abelian ideal of codimension one. In this case the Mostow bundle is a torus bundle over $S^1$ (actually a mapping torus, cf. \cite{BFM}), $G$ is called \emph{almost abelian} and $G$ can be written as a semidirect product $\mathbb{R}\ltimes_{\varphi}\mathbb{R}^n$. The action $\varphi$ of $\RR$ on $\RR^n$ is represented by a family of matrices $\varphi (t)$, which encode the monodromy or ``twist'' in the Mostow bundle (cf. \cite{fisici}). In particular the Lie algebra $\mathfrak{g}$ of $G$ has form $\mathbb{R}\ltimes_{\ad_{X_{n+1}}}\mathbb{R}^n$, where we consider $\mathbb{R}^n $ generated by $\{X_1,...,X_n\}$ and $\mathbb{R}$ by $X_{n+1}$, and $\varphi(t)=e^{t \ad_{X_{n+1}}}$. Moreover, a lattice can be always represented as $\Gamma= \mathbb{Z}\ltimes\mathbb{Z}^n$ (cf. \cite{gor}).

\medskip

In this paper, we find lattices in six dimensional almost abelian solvable Lie groups, using a criterium in \cite{bock} (Proposition~\ref{bocklattice}). 
The cases we deal with correspond to situations when the de Rham cohomology does not agree in general with  the  Chevalley-Eilenberg cohomology $H^* ({\mathfrak g})$ of the Lie algebra ${\mathfrak g}$ of $G$. Namely the \emph{Mostow condition} does not hold (see \cite{Mostow2} and \cite[Corollary 7.29]{Raghunathan} and Section~\ref{cohom}). Intuitively, in these cases there is some extra twist that modifies the topology and it turns out in particular that the cohomology depends on the lattice and not on the solvable Lie algebra only (the latter happens if the Mostow condition holds). We use two methods to compute cohomology and minimal models

\begin{itemize}
\item The \emph{modification of the solvable Lie group} \cite{Guan, CF2} (Section~\ref{cohom}). It consists in modifying the Lie group $G$ into a new $\tilde G$ in such a way that 
 $\tilde G /  \tilde \Gamma$ is diffeomorphic to $G /  \tilde \Gamma$ (where $\tilde \Gamma$ is a finite index subgroup in $\Gamma$, whose algebraic closure is connected) and $H^* (G /  \tilde \Gamma) \cong H^*(\tilde {\mathfrak g})$, where $\tilde {\mathfrak g}$ is the Lie algebra of $\tilde G$.
\item The \emph{Oprea-Tralle method} \cite{OT, OT2}, which consists in applying a result of Felix and Thomas \cite{ft} that yields a Koszul-Sullivan model for non-nilpotent fibrations.
\end{itemize}

We summarize the results in Table~\ref{tabella}, where are listed six dimensional non completely solvable unimodular, almost abelian Lie groups \cite{bock} (see Subsection~\ref{six}) which admit a lattice $\Gamma_{\bar t}$ for some choice of $\bar t\in \RR$ and of the parameters.  We use the same notation as in \cite{bock}. For each of the group in Table~\ref{tabella} we study the \textit{formality} ({\footnotesize \textbf{F}}), \textit{existence of invariant symplectic structures} ({\footnotesize \textbf{IS}}), \textit{existence of not invariant symplectic structures induced by the ones on the modified Lie algebra} ({\footnotesize \textbf{S}}), \textit{Hard Lefschetz property} ({\footnotesize \textbf{HL}}).

\begin{table}[h!]\label{tabella}
\caption{Six dimensional almost abelian solvmanifolds, 
which admit lattices (for some value of the parameter) and do not satisfy the Mostow condition}
\begin{center}
{\scriptsize
\setlength{\extrarowheight}{5pt}
\begin{tabular}{|c|c|c|c|c|c|c|c|}
\hline
\textbf{$G$} & \textbf{$\Gamma_{\bar t}$} & \textbf{$H^*(\mathfrak{g})$} & \textbf{$H^*(G/\Gamma_{\bar t})$} & \textbf{F} & \textbf{IS} & \textbf{S} & \textbf{HL}  \bigstrut 
	\\[-10pt]&&&&&&&\\
\hline \hline
\multirow{2}{*}{$G_{6.8}^{p=0}$} & $\bar t=2\pi$ & \multirow{2}{*}{$b_1=1,b_2=1,b_3=2$} & $b_1=3,b_2=3,b_3=2$ &  Yes & No  & No & $\smallsetminus$ \bigstrut \\ \cline{2-2} \cline{4-8}
 & $\bar t=\pi,\frac{\pi}{2},\frac{\pi}{3}$ &  & $b_1=1,b_2=1,b_3=2$ & Yes & No & $\smallsetminus$ & $\smallsetminus$ \bigstrut \\ 
\hline
\multirow{2}{*}{$G_{6.10}^{a=0}$}& $\bar t=2\pi$ & \multirow{2}{*}{$b_1=2,b_2=3,b_3=4$ }& $b_1=4,b_2=7,b_3=8$ & No & Yes & Yes & No$^\times$ \bigstrut \\ \cline{2-2} \cline{4-8}
& $\bar t=\pi,\frac{\pi}{2},\frac{\pi}{3}$ &  & $b_1=2,b_2=3,b_3=4$ & No & Yes & $\smallsetminus$ & No$^*$ \bigstrut \\
\hline
\multirow{2}{*}{$G_{6.11}^{p=0}$} & $\bar t=2\pi$ & \multirow{2}{*}{$b_1=1,b_2=1,b_3=1$ }&  $b_1=3,b_2=4,b_3=4$ & Yes & No  & No & $\smallsetminus$\bigstrut \\ \cline{2-2} \cline{4-8}
 & $\bar t=\pi,\frac{\pi}{2},\frac{\pi}{3}$ &  & $b_1=1,b_2=1,b_3=1$ & Yes & No & $\smallsetminus$ & $\smallsetminus$ \bigstrut \\
\hline 
\multirow{2}{*}{$G_{5.14}^{0}\times \mathbb{R}$} &$\bar t=2\pi$ & \multirow{2}{*}{$b_1=3,b_2=5,b_3=6$} & $b_1=5,b_2=11,b_3=14$ & No & Yes & Yes & No$^\times$ \bigstrut \\ \cline{2-2} \cline{4-8}
 & $\bar t=\pi,\frac{\pi}{2},\frac{\pi}{3}$ &  & $b_1=3,b_2=5,b_3=6$  & No & Yes & $\smallsetminus$ & No$^*$ \bigstrut \\ 
\hline
\multirow{10}{*}{$G_{5.17}^{p,-p,r}\times \mathbb{R}$} &\multirow{2}{*}{$\bar t=2\pi r_2$} & if $p\neq 0, r \neq \pm 1$ & $p\neq 0:\,b_1=6,b_2=15,b_3=20$ & \multirow{2}{*}{Yes} & \multirow{2}{*}{Yes} & \multirow{2}{*}{Yes} & \multirow{2}{*}{Yes$^\times$} \bigstrut \\ 
& & $b_1=2,b_2=1,b_3=0$ & $p=0:\,b_1=2,b_2=5,b_3=8$ &  & & & \bigstrut \\  \cline{2-2} \cline{4-8}
 & $\bar t=\pi,$ &  & $p\neq 0:\,b_1=2,b_2=1,b_3=0$  & \multirow{2}{*}{Yes} & \multirow{2}{*}{Yes} & \multirow{2}{*}{$\smallsetminus$} & \multirow{2}{*}{Yes$^*$} \bigstrut \\ 
&$r\, \mbox{even}$ &if $p=0,r\neq \pm 1$ & $p=0:\,b_1=4,b_2=7,b_3=8$ &  & & & \bigstrut \\  \cline{2-2} \cline{4-8}
 & $\bar t=\pi, $ & or $p\neq 0, r=\pm 1$ & $p\neq 0:\,b_1=2,b_2=5,b_3=8$  & \multirow{2}{*}{Yes} & \multirow{2}{*}{Yes} & \multirow{2}{*}{$\smallsetminus$} & \multirow{2}{*}{Yes$^*$} \bigstrut \\ 
&$r\, \mbox{odd}$ & $b_1=2,b_2=3,b_3=4$ & $p=0:\,b_1=2,b_2=7,b_3=12$ &  & & & \bigstrut \\  \cline{2-2} \cline{4-8}
$r=\frac {r_1}{r_2} \in \Q$ &$\bar t=\frac{\pi}{2},r\equiv_4 0$&  & $p=0: b_1=4,b_2=7,b_3=8$ & Yes & Yes & $\smallsetminus$ & Yes$^*$ \bigstrut \\  \cline{2-2} \cline{4-8}
&\multirow{2}{*}{$\bar t=\frac{\pi}{2},$}& if $p=0, r=\pm 1$  & $p\neq 0: b_1=2,b_2=3,b_3=4$ & \multirow{2}{*}{Yes} & \multirow{2}{*}{Yes} & \multirow{2}{*}{$\smallsetminus$} & \multirow{2}{*}{Yes$^*$} \bigstrut \\ 
& $r\equiv_4 1,3$ & $b_1=2,b_2=5,b_3=8$ & $p=0: b_1=2,b_2=5,b_3=8$ &  &  &  &  \bigstrut \\ \cline{2-2} \cline{4-8}
&$\bar t=\frac{\pi}{2},r\equiv_4 2$& & $p=0: b_1=2,b_2=3,b_3=4$ & Yes & Yes & $\smallsetminus$ & Yes$^*$ \bigstrut \\ 
\hline
\multirow{3}{*}{$G_{5.18}^{0}\times \mathbb{R}$} &$\bar t=2\pi$ & \multirow{3}{*}{$b_1=2,b_2=3,b_3=4$} & $b_1=4,b_2=9,b_3=13$ & No & Yes & Yes & No$^\times$ \bigstrut \\ \cline{2-2} \cline{4-8}
 & $\bar t=\pi,$ &  & $b_1=2,b_2=5,b_3=8$  & No & Yes & $\smallsetminus$ & No$^*$ \bigstrut \\ \cline{2-2} \cline{4-8}
 & $\bar t= \frac{\pi}{2},\frac{\pi}{3}$ &  & $b_1=2,b_2=3,b_3=4$  & No & Yes & $\smallsetminus$ & No$^*$ \bigstrut \\ 
\hline
\multirow{2}{*}{$G_{3.5}^{0}\times \mathbb{R}^3$} &$\bar t=2\pi$ & \multirow{2}{*}{$b_1=4,b_2=7,b_3=8$} & $b_1=6,b_2=15,b_3=20$ & Yes & Yes & Yes & Yes$^\times$ \bigstrut \\ \cline{2-2} \cline{4-8}
 & $\bar t=\pi,\frac{\pi}{2},\frac{\pi}{3}$ &  & $b_1=4,b_2=7,b_3=8$  & Yes & Yes & $\smallsetminus$ &  Yes$^*$ \bigstrut \\ 
 \hline 
\end{tabular}
$^\times=$ for both the invariant and the not invariant symplectic structures considered. \\
 $^*=$ for the invariant symplectic structures.
}
\end{center}
\label{default}
\end{table}%

Minimal models are computed in Section~\ref{minimal} where we prove the following

\begin{theorem}\label{formality}
$G_{6.8}^{p=0}/\Gamma,\,G_{6.11}^{p=0}/\Gamma,\,G_{5.17}^{p,-p,r}\times\mathbb{R}/\Gamma$ and $G_{3.5}^{0}\times\mathbb{R}^3/\Gamma$ are formal, while $G_{6.10}^{a=0}/\Gamma,\,G_{5.14}^{0}\times\mathbb{R}/\Gamma,\,G_{5.18}^{0}\times\mathbb{R}/\Gamma$ are not formal, for every lattice $\Gamma=\Gamma_{\bar t}$ considered in Table~\ref{tabella}.
\end{theorem}


Some of our results answer to questions  still open  on formality, hard Lefschetz property and cohomology of six dimensional solvmanifolds (see \cite[Proposition 6.18]{bock} and in the decomposable case  \cite[Table 6.3]{bock}).

Note that there are examples where the cohomology depends strongly on the lattice: for example $H^*(G/\Gamma_\pi)\not \cong H^*(G/\Gamma_{2\pi})\not \cong H^*(\cg)$, for $G=G^0_{5.18}\times \RR$.

\medskip

Six dimensional almost abelian solvmanifolds were consider by Andriot, Goi, Minasian and  Petrini \cite{fisici} in string backgrounds where the internal compactification manifold is a
solvmanifold.
Nilmanifolds and solvmanifolds have been extensively used in theoretical physics in type II compactifications, both to four-dimensional
Minkowski or Anti de Sitter,
and appear to be good candidates for possible de Sitter vacua as well.
Indeed their geometry is pretty well understood and, in particular, they can have negative curvature and therefore support internal fluxes (as well as  D-branes and O-plane sources).
In \cite{fisici} is carried out  a discussion of solutions of the supersymmetry (SUSY) equations, and the twist construction of solvmanifolds which serve as internal spaces. In our paper we try to fill out the limitation on solutions they observe, due to lack of isomorphism between the cohomology groups $H^{*}(\cg)$ and $H^{*}(G/\Gamma)$   for non completely solvable manifolds (and, more specifically, for solvmanifolds not satisfying the Mostow condition).

By \cite{FU}, solutions of the supersymmetry (SUSY) equations IIA possess a symplectic
half-flat structure, whereas solutions of the SUSY equations IIB admit a half-flat
structure (see e.g. \cite{CS} for the definition of half-flat structure, cf. also Section~\ref{sympl}). In Section~\ref{sympl}, we prove the following

\begin{prop}\label{not-inv}
We have the following behavior concerning half-flatness of (invariant) symplectic structures for the above solvmanifolds:
\begin{itemize}
\item $G_{6.10}^{a=0}/\Gamma_{2\pi}$ and $G_{5.14}^{0}\times \mathbb{R}/\Gamma_{2\pi}$ admit (not) invariant symplectic forms which are not half-flat.
\item $G_{5.17}^{p,-p,r}\times \mathbb{R}/\Gamma_{2\pi r_2}$ ($r=\frac{r_1}{r_2}\in \Q$)  admits an invariant symplectic form which is half-flat only for $p\geq 0$ and $r=1$ and it admits a not invariant symplectic form which is half-flat. 
\item $G_{5.18}^{0}\times \mathbb{R}/\Gamma_{2\pi}$ and $G_{3.5}^{0}\times \mathbb{R}^3/\Gamma_{2\pi}$ admit not invariant symplectic forms which are half-flat.
\end{itemize}
\end{prop}

\smallskip

\section{The Mostow bundle and almost abelian solvmanifolds}\label{Mostow}

Let $M=G / \Gamma$ be a solvmanifold and let $N$ be the nilradical of $G$ (of course, $N$ agrees with $G$ if and only if $M$ is a nilmanifold). Then $\Gamma_N := \Gamma \cap N$ is a lattice in $N$, $\Gamma N = N
\Gamma$ is  closed in $G$ and $G / (N \Gamma)=:\T^k$ is a torus.
Thus we have the so-called \emph{Mostow fibration}  \cite{Mostow}:
\[\begin{array}{ccl} \label{Mostowbundle}
N/\Gamma_N =(N\Gamma)/ \Gamma & \hookrightarrow & G/\Gamma \\
 & & \downarrow \\
 & & \T^k=G/(N\Gamma) \end{array}
\]
In six dimensions, the nilradical $\cn$ can have dimension from 3 to  6. Dimension 6 corresponds clearly to nilmanifolds. In the codimension one case the Mostow fibration is simpler.  A connected and simply-connected solvable Lie group $G$ with
nilradical $N$ is called \emph{almost nilpotent} if its nilradical has codimension one.
The group $G$ is then given by the semi-direct product $G=\mathbb{R} \ltimes_{\varphi} N$ of its nilradical with $\mathbb{R}$, where $\varphi$ is some action on $N$ depending on the direction
$\mathbb{R}$
\[ (t_1,x_1) \cdot (t_2,x_2) = (t_1 \cdot t_2, x_1\cdot \varphi(t_1)(x_2))
\qquad   (t_i,x_i )\in \mathbb{R} \times N \ .
\]In general, we label by $t$ the coordinate on  $\mathbb{R}$ and by $X_{n+1}=\partial_t$, $n=\dim N$, the corresponding
vector of the algebra.
From a geometrical point of view, $\varphi(t)$ encodes the monodromy of the Mostow bundle.

An \emph{almost abelian solvable group} is an almost nilpotent group whose nilradical is abelian $N=\mathbb{R}^n$. In this case, the action of $\mathbb{R}$ on $N$ is given by
\[\varphi(t)
=e^{t\ \ad_{X_{n+1}}}.
\]

In general, to find lattices in solvable Lie groups is a hard task. Only a necessary criterium is known, namely  that $G$ is unimodular {\cite[Lemma 6.2]{Mil}}.

A nice feature of almost abelian solvable groups is that there is 
a  criterion on the existence of a lattice \cite{bock}
\begin{prop}\label{bocklattice}
Let $G=\mathbb{R}\ltimes_{\varphi}\mathbb{R}^n$ be almost abelian solvable Lie group. Then $G$ admits a lattice if and only if there exists a $ t_0 \neq 0$ for which $\varphi(t_0)$ can be conjugated
to an integer matrix. 
\end{prop}

We shall call \emph{almost abelian solvmanifold} the corresponding quotient of an almost abelian solvable Lie group by a lattice.

\smallskip

Lattices of solvmanifolds yield their diffeomorphism class. Indeed

\begin{theorem} \cite[Theorem 3.6]{Raghunathan}  \label{diffeo}
Let $G_i / \Gamma_i$ be solvmanifolds for $i \in \{1,2\}$ and
$\psi : \Gamma_1 \to \Gamma_2$ an isomorphism. Then there exists a diffeomorphism $\Psi : G_1 \to G_2$ such that
\begin{itemize}
\item[(i)] $\Psi|_{\Gamma_1} = \psi ,$
\item[(ii)] $\Psi(p \gamma) = \Psi(p) \psi(\gamma)$, for any $\gamma \in \Gamma_1$ and any $p \in G_1$.
\end{itemize}
\end{theorem}

As a consequence  two compact solvmanifolds with isomorphic fundamental groups are diffeomorphic.


\section{Modification of the cohomology}\label{cohom}

If the algebraic closures ${\mathcal A} (\Ad_G (G))$ and  ${\mathcal A} (\Ad_G (\Gamma))$ are equal, one says that $G$ and $\Gamma$ satisfy the \emph{Mostow condition} (see \cite{Raghunathan} for more details and definitions). In this case, the de Rham  cohomology $H^*(M)$ of  the compact  solvmanifold $M = G / \Gamma$  can be computed  by the  Chevalley-Eilenberg cohomology $H^* ({\mathfrak g})$ of the Lie algebra ${\mathfrak g}$ of $G$ (see \cite{Mostow2} and \cite[Corollary 7.29]{Raghunathan}); indeed, one has the isomorphism $H^*(M) \cong H^* ({\mathfrak g})$.
A special case is provided by nilmanifolds (Nomizu's Theorem, \cite{Nomizu}) and more generally if $G$ is \emph{completely solvable} \cite{Hattori}, i.e. all the linear operators $\ad_X: {\mathfrak g} \to {\mathfrak g}$, $X \in {\mathfrak g}$  have only real eigenvalues.

\smallskip

For almost abelian solvmanifolds, Gorbatsevich found a criterion to decide whether the Mostow condition holds \cite{gor}:

\begin{prop} \label{gorba}
The Mostow condition is satisfied if and only if $\pi i$ can not be written as linear combination in $\mathbb{Q}$ of the eigenvalues of $t_0\ad_{X_{n+1}}$, where $\Gamma$ is generated by $t_0$.
\end{prop}

Let $M=G/\Gamma$ be a solvmanifold.
By \cite[Theorem 6.11, p. 93]{Raghunathan} it is not restrictive to suppose that 
 ${\mathcal A} (\Ad_G (\Gamma))$ is connected. Otherwise we  pass from $\Gamma$ to a finite index subgroup  $\tilde \Gamma$. This  is equivalent to passing from $M = G/ \Gamma$  to the space $ G/ \tilde \Gamma$ which is a finite-sheeted covering of $M$. 
 
By Borel density theorem (see e.g. \cite[Theorem 5.5]{Raghunathan}), there exists a compact torus $\T_{cpt}$ such that  $\T_{cpt}  {\mathcal A} (\Ad_G (\Gamma)) = {\mathcal A} (\Ad_G (G))$. 

Then the main step for the ``modification method'' is the following

\begin{theorem} \cite{CF2} \label{metodo}
Let $G$ be a solvable simply connected Lie group and let $\Gamma$ be lattice in $G$ such that $G/\Gamma$ is a solvmanifold and ${\mathcal A} (\Ad_G (\Gamma))$ is connected. Suppose ${\mathcal A} (\Ad_G (G))=\T_{c}{\mathcal A} (\Ad_G (\Gamma))$, with $\T_{c}$ maximal compact torus of ${\mathcal A} (\Ad_G (G))$. Then exists a normal simply connected subgroup  $\tilde{G}$ of $\T_{c}\ltimes G$ such that ${\mathcal A} (\Ad_{\tilde{G}} (\tilde{G}))={\mathcal A} (\Ad_{\tilde{G}} (\Gamma)) $.
\end{theorem}

The Mostow condition holds for the Lie group $\tilde{G}$, so $H^*(\tilde{G}/\Gamma)=H^*(\tilde{\mathfrak{g}})$.

The modified solvable group $\tilde{G}$ is obtained from $G$ by killing the action of subtorus $\SSS_c$ that we get by comparing the compact and $\mathbb{C}$-diagonalizable part of ${\mathcal A} (\Ad_G (G))$ and ${\mathcal A} (\Ad_G (\Gamma))$. More precisely,  let  $\overline{\SSS}_{c}$  be a maximal  compact torus of ${\mathcal A} (\Ad_G (\tilde \Gamma))$ contained in $\T_c$.  Let $\SSS_{c}$ be a subtorus of $\T_{c}$ complementary to $\overline{\SSS}_{c}$ so that $\T_{c} = \SSS_{c} \times \overline{\SSS}_{c}$. Let $\sigma$  be the composition of the homomorphisms:
$$
\sigma: G  \stackrel{\Ad}  \longrightarrow {\mathcal A} (\Ad_G (G))  \stackrel{{\mbox {\footnotesize proj}}}  \longrightarrow \T_{c}  \stackrel{{\mbox {\footnotesize proj}}} \longrightarrow \SSS_{c} \stackrel{ x \to x^{-1}}  \longrightarrow \SSS_{c}.
$$
One uses $\sigma$ to get rid of $\SSS_{c}$ (see \cite{CF2}).

It turns out that $\tilde{G}$ is diffeomorphic to $G$, they are both simply connected and, by Theorem~\ref{diffeo}, $\tilde{G}/\Gamma$ is diffeomorphic to $G/\Gamma$. Thus $H^*(G/\Gamma) =H^*( \tilde{G}/\Gamma)$ and we get

\begin{cor}\label{modif}
Let $G$ be a solvable simply connected Lie group and let $\Gamma$ be lattice in $G$ such that $G/\Gamma$ is a solvmanifold and ${\mathcal A} (\Ad_G (\Gamma))$ is connected. Then we have $$H^*(G/\Gamma)=H^*(\tilde{G}/\Gamma)=H^*(\tilde{\mathfrak{g}}),$$ where $\tilde{\mathfrak{g}}$ is the Lie algebra of $\tilde{G}$.
\end{cor}

Observe that the lattice $\Gamma$ is not modified and, indeed, as remarked, $G/\Gamma$ is an Eilenberg-MacLane space, so that its topology is determined by its fundamental group $\Gamma$ only.

\begin{remark}
The Lie algebra  $\tilde {\mathfrak g}$ of $ \tilde G$ can be identified by
$$
\tilde {\mathfrak g} = \{ (X_{\mathfrak s}, X) \mid X \in {\mathfrak  g} \}
$$
with  Lie bracket:
$$
[( X_{\mathfrak s}, X), (Y_{\mathfrak s}, Y)] = (0, [X, Y] - \ad(X_{\mathfrak s})  (Y) +\ad( Y_{\mathfrak s}) (X)).
$$
where $X_{\mathfrak s}$ the image  $\sigma_* (X)$, for $X \in {\mathfrak g}$. (see \cite[Proposition 6.1]{CF2})
\end{remark}

\medskip

In the general case for a lattice $\Gamma$, the method runs as follows. Given $M = G / \Gamma$, there exists a  finite covering space  $\tilde M =  G / \tilde \Gamma$, i.e., $\Gamma/ \tilde \Gamma$ is a  finite group, such that  for $\tilde \Gamma$ it holds that ${\mathcal A}(\Ad_G (\tilde \Gamma)$ is connected.

Thus $H^* (G /  \Gamma) \cong H^* (G /  \tilde \Gamma)^{\Gamma/\tilde \Gamma}$ (the invariants by the action of the finite group $\Gamma/\tilde \Gamma$). 

\smallskip

\begin{remark}\label{incl}
There is a natural injection $H^*(\cg) \hookrightarrow H^* (G /  \Gamma)$, \cite[Theorem 3.2.10]{OT}.  Hence cohomology classes in $H^*(\cg)$ correspond to cohomology classes of invariant differential forms.
\end{remark}
\medskip

\subsection{Six dimensional almost abelian Lie groups}\label{six}

We are interested in six dimensional, unimodular 
almost abelian Lie groups which are not completely solvable. There are eleven such Lie groups that can admit a lattice and their Lie algebras are the following \cite{bock}:
\begin{description}[leftmargin=2.8cm,
style=nextline]
\itemsep3pt \parskip3pt \parsep0pt 
\item[{\rm $ \mathfrak{g}_{6.8}^{a,b,c,p}$}] $[X_1,X_6]=aX_1, \; [X_2,X_6]=bX_2, \; [X_3,X_6]=cX_3, \; [X_4,X_6]=pX_4-X_5,$ \\ $ [X_5,X_6]=X_4+pX_5, \qquad  a+b+c+2p=0, \quad 0< |c| \leq |b| \leq |a| $. 
\item[{\rm$ \mathfrak{g}_{6.9}^{a,b,p}$}] $[X_1,X_6]=aX_1, \; [X_2,X_6]=bX_2, \; [X_3,X_6]=X_2+bX_3, $ \\ $[X_4,X_6]=pX_4-X_5, \; [X_5,X_6]=X_4+pX_5, \qquad  a+2b+2p=0, \quad a\neq 0 $.
\item[{\rm$ \mathfrak{g}_{6.10}^{a,-\frac{3}{2}a}$}] $[X_1,X_6]=aX_1, \; [X_2,X_6]=X_1+aX_2, \; [X_3,X_6]=X_2+aX_3, $ \\ $[X_4,X_6]=-\frac{3}{2}aX_4-X_5,\; [X_5,X_6]=X_4-\frac{3}{2}aX_5$.
\item[{\rm$ \mathfrak{g}_{6.11}^{a,p,q,s}$}] $[X_1,X_6]=aX_1, \; [X_2,X_6]=pX_2-x_3, \; [X_3,X_6]=X_2+pX_3, $ \\ $ [X_4,X_6]=qX_4-sX_5,\;  [X_5,X_6]=sX_4+qX_5, \quad  a+2p+2q=0, \quad as \neq 0 $.
\item[{\rm$ \mathfrak{g}_{6.12}^{-4p,p}$}] $[X_1,X_6]=-4pX_1, \; [X_2,X_6]=pX_2-X_3, \; [X_3,X_6]=X_2+pX_3, $ \\ $[X_4,X_6]=X_2+pX_4-X_5,\; [X_5,X_6]=X_3+X_4+pX_5, \qquad  p \neq 0 $.
\item[{\rm$ \mathfrak{g}_{5.13}^{-1-2q,q,r}\oplus \mathbb{R}$}] $[X_1,X_5]=X_1, \; [X_2,X_5]=(-1-2q)X_2, \; [X_3,X_5]=qX_3-rX_4,$ \\ $ [X_4,X_5]=rX_3+qX_4,\quad  q \neq -\frac{1}{2},\quad r \neq 0, \quad -1 \leq q \leq 0 $. 
\item[{\rm$ \mathfrak{g}_{5.14}^{0}\oplus \mathbb{R}$}] $ [X_2,X_5]=X_1, \; [X_3,X_5]=-X_4,\; [X_4,X_5]=X_3$.
\item[{\rm $ \mathfrak{g}_{5.17}^{p,-p,r}\oplus \mathbb{R}$}] $[X_1,X_5]=pX_1-X_2, \; [X_2,X_5]=X_1+pX_2, \; [X_3,X_5]=-pX_3-rX_4, $ \\ $[X_4,X_5]=rX_3-pX_4, \quad r \neq 0$.
\item[{\rm$ \mathfrak{g}_{5.18}^{0}\oplus \mathbb{R}$}] $[X_1,X_5]=-X_2, \; [X_2,X_5]=X_1, \; [X_3,X_5]=X_1-X_4,\; [X_4,X_5]=X_2+X_3 $.
\item[{\rm$ \mathfrak{g}_{4.6}^{-2p,p}\oplus \mathbb{R}^2$}] $[X_1,X_4]=-2pX_1, \; [X_2,X_4]=pX_2-X_3, \; [X_3,X_4]=X_2+pX_3, \quad  p>0 $.
\item[{\rm$ \mathfrak{g}_{3.5}^{0}\oplus \mathbb{R}^3$}] $[X_1,X_3]=-X_2, \; [X_2,X_3]=X_1 $.
\end{description}

\smallskip

Next, we apply Proposition~\ref{bocklattice} to determine for which values of $t=\bar t$,  $\varphi(\bar t)=\exp (\bar t \, \ad_{X_6})$ determines  a lattice $\Gamma_{\bar t}$ in $G$.

Note in particular that as consequence of Proposition~\ref{bocklattice}, both the characteristic polynomial and the minimal polynomial of $\exp (\bar t \, \ad_{X_{n+1}})$ must have integer coefficients.

\smallskip

To perform the computations we use the Maple software.

To illustrate the method, we develop in detail the case of $G_{6.8}^{a,b,c,p}$, writing down the results in the other cases.
\bigskip

\noindent $\bullet$ {\large{$G_{6.8}^{a,b,c,p}$}}: 
From the structure equations of $\cg_{6.8}^{a,b,c,p}$ (Section~\ref{cohom}), we get
\begin{displaymath}
 \exp(t \, \ad_{X_6})=
\left( \begin{array}{ccccc}
e^{t(-b-c-2p)} & 0 & 0 & 0 & 0 \\
0 & e^{tb} & 0 & 0 & 0 \\
0 & 0 & e^{tc} & 0 & 0 \\
0 & 0 & 0 & e^{tp}\cos t & e^{tp}\sin t \\
0 & 0 & 0 & -e^{tp}\sin t & e^{tp}\cos t
\end{array} \right)
\end{displaymath}
The eigenvalues of $t \, \ad_{X_6}$ are 
$$
\pm t \,  i, \; t \,  b,\; t \,  c, -(b+c)t \, \, ,
$$
so the Mostow condition does not hold for $t=\bar t$ a rational multiple of $\pi$.
To apply Corollary~\ref{modif} we need to have ${\mathcal A} (\Ad_G (\Gamma_{\bar t}))$ connected. Using the same arguments as in the proof of Proposition~\ref{gorba}, (see \cite{gor}) one can see this is the case for $\bar t = 2 \pi$.
Indeed, using the Jordan decomposition into semisimple and nilpotent parts, one gets that the only blocks whose algebraic closures are not in general connected are the subgroups given by the exponentials of the roots of the complex eigenvalues. In this case they are the cyclic subgroups  $\left( \begin{array}{cc}
\cos (n \bar t) & \sin (n \bar t) \\
-\sin (n \bar t) & \cos (n \bar t)
\end{array} \right)
$, $n \in \ZZ$, for $t=\bar t$ a rational multiple of $\pi$.
The above cyclic subgroups are connected only if it they are trivial, i.e., for $\bar t = 2 \pi$.

\smallskip

Let us consider then $\Gamma_{\bar t}$ for $\bar t=2 \pi$.

Setting $e^{2\pi b}=w, \; e^{2\pi c}=v, \; e^{-2\pi p}=k$,  we have
\begin{displaymath}
\exp(2\pi \, \ad_{X_6}) =
\left( \begin{array}{ccccc}
\frac{k^2}{wv} & 0 & 0 & 0 & 0 \\
0 & w & 0 & 0 & 0 \\
0 & 0 & v & 0 & 0 \\
0 & 0 & 0 & \frac{1}{k} & 0 \\
0 & 0 & 0 & 0 & \frac{1}{k}
\end{array} \right)
\end{displaymath}
Its minimal polynomial is 
{\small
\begin{eqnarray*}
\MinPol(x)& = & k-\frac{k^3wv+wk^2+k^2v+v^2w^2}{wvk}x+\frac{vk^3+wk^3+w^2v^2k+k^2+wv^2+w^2v}{wvk}x^2+\\
& & -\frac{k^3+kwv^2+kw^2v+w}{wvk}x^3+x^4 
\end{eqnarray*}
}
So, it can have integer coefficients only if $k \in \mathbb{Z}$.

We set $w+v=r, \; wv=s$ and the coefficients $p_i$ of $x^i$ in $\MinPol(x)$ become:
\[
\begin{array}{l}
p_1=\dfrac{k^3s+k^2r+s^2}{ks}=k^2+\dfrac{k^2r+s^2}{ks}\, , \\
p_2= \dfrac{k^3r+kr^2+k^2+rs}{ks}\, ,  \qquad \quad p_3= \dfrac{k^3+krs+s}{ks}\, .\end{array}
\]
Hence $p_1 \in \mathbb{Z}\,$ if and only if $\,q_1=\frac{k^2r+s^2}{ks} \in \mathbb{Z}$ and $p_2-kq_1=\frac{k^2+rs}{ks}$.
If $p_1,p_2 \in \mathbb{Z}$ then $h:=p_2-kq_1 \in \mathbb{Z}$. $s=\frac{k^2}{hk-r}$, then $p_3=\frac{hk^2+1}{k}= hk+\frac{1}{k}$.  So $p_3 \in \mathbb{Z}\,$ if and only if $\frac{1}{k} \in \mathbb{Z} $, but $k \in \mathbb{Z}$, so $k=1$ and $p=0$.

\emph{Therefore for $p\neq 0$, $\Gamma_{2\pi}$ is not a lattice}.

\smallskip

\emph{Next we check the existence of a lattice for $p=0$.} The characteristic polynomial of $\exp(2\pi \, \ad_{X_6})$ has coefficients
\[
\begin{array}{l}
a_0=-1\quad a_1=2+\dfrac{r+s^2}{s}\quad a_2=-1-\dfrac{2s^2+2r+rs+1}{s}=-1-2\dfrac{s^2+r}{s}-\dfrac{rs+1}{s}\\[-8pt]\\
a_3=1+\dfrac{2rs+2+s^2+r}{s}= 1+2\dfrac{rs+1}{s} +\dfrac{s^2+r}{s}\qquad a_4=-2-\dfrac{rs+1}{s}\end{array}
\]
So $a_1,\, a_2, \, a_3, \,a_4 \in \mathbb{Z}$ if and only if $\dfrac{s^2+r}{s}, \, \dfrac{rs+1}{s} \in \mathbb{Z}$ and we must check that the solutions are such that $w$ and $v$ are positive. To this goal we consider the system
\begin{equation}\label{sys}
\left\{ \begin{array}{l}
\dfrac{s^2+r}{s}=h_1 \\[-8pt]\\
\dfrac{rs+1}{s}=h_2 \\[-8pt]\\
r>0\\[-8pt]\\
0<s\leq \frac{r^2}{4}
\end{array} \right.
\end{equation}
that admits solutions for some values of the integers $h_1$ and $h_2$ (for example for $h_1=5, h_2=6$). In particular we can not accept the solutions $\{s=r-1\}$, because they correspond to $b=0$ or $c=0$ and $\{s=1\}$, because it corresponds to $a=0$.

\medskip 
Thus, for $p=0$, we can find values of $b$ and $c$ (and $a=-b-c$) such that the characteristic polynomial of $\exp(2\pi  \ad_{X_6})$ has integer coefficients and we can check by direct computation that $\exp(2 \pi \, \ad_{X_6})$ is conjugate to 
$\left( \begin{array}{ccccc}
0 & 0 & 1 & 0 & 0 \\
1 & 0 & -h_1 & 0 & 0 \\
0 & 1 & h_2 & 0 & 0 \\
0 & 0 & 0 & 1 & 0 \\
0 & 0 & 0 & 0 & 1
\end{array} \right)$. 
Therefore, \emph{for some choice of the parameters $b$ and $c$, $\Gamma_{2\pi}$ is a lattice}. We denote the group $G_{6.8}^{a,b,c,0}$ for the above choices of the parameters $a, b, c$ by $G_{6.8}^{p=0}$ for short.

\medskip
Next we verify the Mostow condition: the eigenvalues of $2\pi \ad_{X_6}$ are 
$$
\pm 2\pi i, \; 2\pi b,\; 2\pi c,\newline -(b+c)2\pi\, ,
$$
so we can easily find a linear combination in $\mathbb{Q}$ that gives $\pi i$. Hence, by Proposition~\ref{gorba} the Mostow condition does not hold. 

To compute the cohomology we have then to apply the modification method.

The Lie group $G_{6.8}^{p=0}$ is defined by the map
\begin{displaymath}
\exp(t \, \ad_{X_6})=
\left( \begin{array}{ccccc}
e^{t(-b-c)} & 0 & 0 & 0 & 0 \\
0 & e^{tb} & 0 & 0 & 0 \\
0 & 0 & e^{tc} & 0 & 0 \\
0 & 0 & 0 & \cos t & \sin t \\
0 & 0 & 0 & -\sin t & \cos t
\end{array} \right)
\end{displaymath}
By definition the subtorus $\SSS_c$ is the compact part of the $\mathbb{C}$-diagonalizable one, that is the product of $\SSS_c$ and the $\mathbb{R}$-diagonalizable torus, so it is just the circle give by the block 
$\left( \begin{array}{cc}
\cos t & \sin t \\
-\sin t & \cos t
\end{array} \right)
$ and then $\tilde{G}_{6.8}^{p=0}$ is defined by
$\left( \begin{array}{ccccc}
e^{t(-b-c)} & 0 & 0 & 0 & 0 \\
0 & e^{tb} & 0 & 0 & 0 \\
0 & 0 & e^{tc} & 0 & 0 \\
0 & 0 & 0 & 1 & 0 \\
0 & 0 & 0 & 0 & 1
\end{array} \right)$.

Hence, the structure constants of $\tilde{\cg}_{6.8}^{p=0}$ are $$[X_1,X_6]=-(b+c)X_1, \; [X_2,X_6]=bX_2, \; [X_3,X_6]=cX_3.$$
Thus \emph{$\tilde{\mathfrak{g}}_{6.8}^{p=0}$ is isomorphic to the decomposable solvable Lie algebra $\cg_{4.5}^{k, -k-1}\oplus \RR^2$ (for some $k$)}, cf. \cite[Appendix A]{bock}.

By Corollary~\ref{modif} the cohomology of $G_{6.8}^{p=0}/\Gamma_{2\pi}$ is given by the cohomology groups of the Lie algebra $\tilde {\mathfrak{g}}_{6.8}^{p=0}$, with $(\tilde {\mathfrak{g}}_{6.8}^{p=0})^*= \langle \alpha^1,..., \alpha^6\rangle$, where $\alpha^i$ are the dual forms of $X_i$ ($i=1, \dots, 6$). Thus they are:
\begin{displaymath}
\begin{array}{l}
H^1(G_{6.8}^{p=0}/\Gamma_{2\pi})=H^1(\mathfrak{\tilde{g}}_{6.8}^{p=0})=\langle \alpha^4,\,\alpha^5,\,\alpha^6\rangle \\ [-8pt]\\
H^2(G_{6.8}^{p=0}/\Gamma_{2\pi})=H^2(\mathfrak{\tilde{g}}_{6.8}^{p=0})=\langle \alpha^{45},\,\alpha^{46},\,\alpha^{56}\rangle \\[-8pt]\\
H^3(G_{6.8}^{p=0}/\Gamma_{2\pi})=H^3(\mathfrak{\tilde{g}}_{6.8}^{p=0})=\langle \alpha^{123},\,\alpha^{456}\rangle
\end{array}
\end{displaymath}

Here and in the sequel, for the sake of symplicity, we do not use any special symbol for the cohomology class, writing down one of its representatives.
\smallskip

Next, \emph{let us investigate if there are integer values $k$ such that $\Gamma_{2\pi/k}$ is a lattice in $G_{6.8}^{p=0}$}.
\begin{displaymath}
\exp(2\pi/k\, \ad_{X_6}) =
\left( \begin{array}{ccccc}
e^{2\pi(-b-c)/k} & 0 & 0 & 0 & 0 \\
0 & e^{2\pi b/k} & 0 & 0 & 0 \\
0 & 0 & e^{2\pi c/k} & 0 & 0 \\
0 & 0 & 0 & \cos 2\pi /k & \sin 2\pi /k \\
0 & 0 & 0 & -\sin 2\pi /k & \cos 2\pi /k
\end{array} \right)
\end{displaymath}
We set $e^{2\pi b/k}=w, \; e^{2\pi c/k}=v, \;\cos 2\pi /k=u/2 $, and $w+v=r, \; wv=s$. Then the coefficients of the characteristic polynomial of $\exp(2\pi/k\, \ad_{X_6})$ become:
\[
\begin{array}{l}
a_1= \dfrac{us+r+s^2}{s}=u+\dfrac{r+s^2}{s} \qquad a_2= -1-\dfrac{ur+us^2+1+rs}{s} \\
a_3=1+ \dfrac{u+urs+s^2+r}{s} \qquad a_4=-\dfrac{1+rs+us}{s}=-\dfrac{1+rs}{s}-u
\end{array}
\]
Thus $a_2=-ua_1+a_4+u+u^2-1$ and $a_3=-ua_4+a_1-u-u^2+1$, so if $a_1,a_2,a_3,a_4 \in \mathbb{Z}$, then $a_1+a_4$ and $a_2+a_3$ are integer and so $u \in \mathbb{Q}$. Therefore, if $\cos 2\pi /k$ is not rational, then $\Gamma_{2\pi /k}$ is not a lattice.

If $u \in \mathbb{Q}$, then the characteristic polynomial has integer coefficients if and only if the same system \eqref{sys} as the one found for  $\bar t=2\pi$ admits a solution. Again by direct computation we check that the matrix $A$ is conjugate with $\exp(\bar t \, \ad_{X_6})$, for every $\bar t$ such that $\cos \bar t= \pm 1, 0, \pm \frac{1}{2}$.

Hence \emph{we have a lattice in $G_{6.8}^{p=0}$ for 
$\bar t=\frac{2\pi}{k}$ such that $\cos \bar t= \pm 1, 0, \pm \frac{1}{2}$.}

We compute the cohomology groups by finding the invariants of the action of $\Gamma_{\bar t}/\Gamma_{2\pi}$ for $\bar t=\frac{\pi}{2},\frac{\pi}{3}, \pi $. For the other cases we get the same result for the cohomology.

For $\bar t=\frac{\pi}{2}$, let 
\begin{displaymath}
\psi_2:=\exp(\pi/2\, \ad_{X_6})^t =
\left( \begin{array}{ccccc}
e^{\pi(-b-c)/2} & 0 & 0 & 0 & 0 \\
0 & e^{\pi b/2} & 0 & 0 & 0 \\
0 & 0 & e^{\pi c/2} & 0 & 0 \\
0 & 0 & 0 & 0 & -1 \\
0 & 0 & 0 & 1 & 0
\end{array} \right) .\end{displaymath}
Hence
\[
\begin{array}{ll}
\psi_2 \alpha^4=\alpha^5, \quad \psi_2 \alpha^5=-\alpha^4, \quad \psi_2 \alpha^6=\alpha^6 & \Longrightarrow\quad H^1(G_{6.8}^{p=0}/\Gamma_{\pi /2})=\langle \alpha^6 \rangle\, , \\[-8pt]\\
\psi_2 \alpha^{45}=\alpha^{45}, \quad \psi_2 \alpha^{46}=\alpha^{56}, \quad \psi_2 \alpha^{56}=-\alpha^{46} & \Longrightarrow\quad H^2(G_{6.8}^{p=0}/\Gamma_{\pi /2})=\langle \alpha^{45} \rangle\, ,\\[-8pt]\\
\psi_2 \alpha^{123}=\alpha^{123}, \quad \psi_2 \alpha^{456}=\alpha^{456}& \Longrightarrow\quad H^3(G_{6.8}^{p=0}/\Gamma_{\pi /2})=\langle \alpha^{123}, \;\alpha^{456} \rangle\, .
\end{array}\\[-8pt]\\
\]
Similarly, one gets
\[
\begin{array}{l}
 H^1(G_{6.8}^{p=0}/\Gamma_{\pi /3})=\langle \alpha^6 \rangle\, ,\\[-8pt]\\
H^2(G_{6.8}^{p=0}/\Gamma_{\pi /3})=\langle \alpha^{45} \rangle\, ,\\[-8pt]\\
H^3(G_{6.8}^{p=0}/\Gamma_{\pi /3})=\langle \alpha^{123}, \;\alpha^{456} \rangle\, .
\end{array}\\[-8pt]
\qquad
\begin{array}{l}
H^1(G_{6.8}^{p=0}/\Gamma_{\pi})=\langle \alpha^6 \rangle\\[-8pt]\\
H^2(G_{6.8}^{p=0}/\Gamma_{\pi})=\langle \alpha^{45} \rangle\\[-8pt]\\
H^3(G_{6.8}^{p=0}/\Gamma_{\pi })=\langle \alpha^{123}, \;\alpha^{456} \rangle 
\end{array}\\[-8pt]
\]

\medskip

\noindent $\bullet$ {\large{$G_{6.9}^{a,b,p}$}}: 
 again $\exp (t \, \ad_{X_6})$ has a pair complex conjugate roots. Thus one would get a lattice $\Gamma_{\bar t}$ for which the Mostow condition does not hold and 
 ${\mathcal A} (\Ad_G (\Gamma_{\bar t}))$ is connected for $\bar t=2 \pi$. However, 
 one can show that \emph{there is no lattice for $\bar t=2\pi$}.

\medskip

\noindent $\bullet$ {\large{$G_{6.10}^{a}$}}: $\exp (t \, \ad_{X_6})$ has a pair complex conjugate roots. \emph{If $a \neq 0$ there is no lattice $t=2\pi$, but $\Gamma_{2\pi}$ is a lattice for $G_{6.10}^0$. }
The eigenvalues of $2\pi \ad_{X_6}$ are $\pm 2\pi i$, so the Mostow condition does not hold. Then $\tilde{\cg}_{6.10}^{0}$ has structure constants $[X_2,X_6]=X_1, \; [X_3,X_6]=X_2$ and \emph{is isomorphic to the nilmanifold $\cg_{4.1}\oplus \RR^2$}, cf. \cite[Appendix A]{bock}. The cohomology groups of $G_{6.10}^{0}/\Gamma_{2\pi}$ are
 \begin{displaymath}
\begin{array}{l}
H^1(G_{6.10}^{0}/\Gamma_{2\pi})=H^1(\mathfrak{\tilde{g}}_{6.10}^{0})=\langle \alpha^3,\,\alpha^4,\,\alpha^5,\,\alpha^6\rangle \\ [-8pt]\\
H^2(G_{6.10}^{0}/\Gamma_{2\pi})=H^2(\mathfrak{\tilde{g}}_{6.10}^{0})=\langle \alpha^{16},\,\alpha^{23},\,\alpha^{34},\,\alpha^{35},\,\alpha^{45},\,\alpha^{46},\,\alpha^{56}\rangle \\[-8pt]\\
H^3(G_{6.10}^{0}/\Gamma_{2\pi})=H^3(\mathfrak{\tilde{g}}_{6.10}^{0})=\langle \alpha^{123},\,\alpha^{126},\,\alpha^{146},\,\alpha^{156},\,\alpha^{234},\,\alpha^{235},\,\alpha^{345},\,\alpha^{456}\rangle
\end{array}\\[-8pt]
\end{displaymath}

\smallskip

\noindent The subgroups $\Gamma_{2\pi/k}$ ($k \in \ZZ$) are also lattices if and only if $2\cos\left( \frac{2\pi}{k} \right) \in \mathbb{Z}$.
  In all these cases we have
\begin{displaymath}
\begin{array}{l}
H^1(G_{6.10}^{0}/\Gamma_{\bar t})=\langle \alpha^3,\,\alpha^6 \rangle \\[-8pt]\\
H^2(G_{6.10}^{0}/\Gamma_{\bar t})=\langle \alpha^{16},\,\alpha^{23},\,\alpha^{45} \rangle
\\[-8pt]\\
H^3(G_{6.10}^{0}/\Gamma_{\bar t})=\langle \alpha^{123}, \;\alpha^{126}, \;\alpha^{345}, \;\alpha^{456} \rangle .
\end{array}\\[-8pt]
\end{displaymath}

\smallskip

\begin{remark} 
The lattice $\Gamma_{\pi}$ was found in \cite[Proposition 6.18]{bock}. In part (ii) it is stated that if there is a lattice in $G_{6.10}^{0}$ such that the corresponding solvmanifold satisfies
$b_1 = 2$ and $b_2 = 3$, then it is symplectic and not formal. Here we show that, for  example, $\Gamma_{\pi}$, is such a lattice. We will deal about symplectic structures and formally later (Section~\ref{sympl}).
\end{remark}

\noindent $\bullet$ {\large{{$G_{6.11}^{a,p,q,s}$}}: 
in this case from the structure equations of $\cg_{6.11}^{a,p,q,s}$ we get
\begin{displaymath}
\ad_{X_6}=
\left( \begin{array}{ccccc}
-2(p+q) & 0 & 0 & 0 & 0 \\
0 & p & 1 & 0 & 0 \\
0 & -1 & p & 0 & 0 \\
0 & 0 & 0 & q & s \\
0 & 0 & 0 & -s & q
\end{array} \right)\, .
\end{displaymath}
Thus we have two non diagonal blocks with a couple complex conjugate roots. Hence there could be more situations for which the Mostow condition does not hold.

\smallskip

\noindent (i): \emph{If $s \in \mathbb{Q}$ (say $s=\frac{s_1}{s_2}$), then $\Gamma_{2\pi s_2}$ would be the right choice of parameter to have ${\mathcal A} (\Ad_G (\Gamma_{2\pi s_2}))$ connected. But $\Gamma_{2\pi s_2}$ is not a lattice in $G_{6.11}^{a,p,q,s}$.}
\begin{proof} Exponentiating we get
\begin{displaymath}
\exp(2\pi s_2\, \ad_{X_6}) =
\left( \begin{array}{ccccc}
e^{-4(p+q)\pi s_2} & 0 & 0 & 0 & 0 \\
0 & e^{2p\pi s_2} & 0 & 0 & 0 \\
0 & 0 & e^{2p\pi s_2} & 0 & 0 \\
0 & 0 & 0 & e^{2q\pi s_2} & 0 \\
0 & 0 & 0 & 0 & e^{2q\pi s_2}
\end{array} \right)\, .
\end{displaymath}
We set $ e^{-2(p+q)\pi s_2}=\alpha$ and $ e^{-2q\pi s_2}+e^{-2p\pi s_2}=\beta$, so its minimal polynomial is
\begin{eqnarray*}
\MinPol(x) = -\alpha+\frac { \left( \alpha^2\beta+1 \right) x}{\alpha}-\frac 
{ \left( \alpha^3+\beta \right) {x}^{2}}{\alpha}+x^3
\end{eqnarray*}
so it can have integer coefficients only if $\alpha \in \mathbb{Z}$. Then $\frac{\alpha^2\beta+1}{\alpha}=\beta+\frac{1}{\alpha} \in \mathbb{Z}$ implies $\beta  \in \mathbb{Q}$.
But then $\frac{\alpha^3+\beta}{\alpha}=\alpha^2+\frac{\beta}{\alpha} \in \mathbb{Z}$ implies $\frac{\beta}{\alpha} \in \mathbb{Z}$ and so $\beta \in \mathbb{Z}$.
Therefore if $\alpha$ and $\beta$ are not both integer we have no lattice $\Gamma_{2\pi s_2}$.
Suppose $\alpha,\beta \in \mathbb{Z}$, then $\beta+\frac{1}{\alpha} \in \mathbb{Z}$ only if $\alpha=1$ that is $a=p+q=0$, but this value is not acceptable, so $\Gamma_{2\pi s_2}$ is not a lattice.
\end{proof}

\noindent (ii): 
\emph{
If $s$ is irrational  then one can look for lattices $\Gamma_{\bar t}$ with ${\mathcal A} (\Ad_G (\Gamma_{\bar t}))$ connected for $\bar t =2 \pi$. For $p \neq 0$ there is no lattice for $\bar t=2\pi$, but $\Gamma_{2\pi}$ is a lattice for $G_{6.11}^{a,0,q,s}$ for some value of $q$ and $s$ (recall, $a+2p+2q=0$). We denote the group $G_{6.11}^{a,0,q,s}$ for these choices of the parameters by $G_{6.11}^{p=0}$
} 

The Mostow condition does not hold and by modification method we get that 
$\mathfrak{\tilde{g}}_{6.11}^{p=0}$ has structure equations $[X_1,X_6]=-2qX_1, \; [X_4,X_6]=qX_4-sX_5, \; [X_5,X_6]=sX_4+qX_5 $, so it is isomorphic to $\cg_{4.6}^{-2k, k}\oplus \RR^2$ for some $k$, cf.  \cite[Appendix A]{bock}. The Lie algebra $\mathfrak{\tilde{g}}_{6.11}^{p=0}$ is not completely solvable, but satisfies the Mostow condition for our choice of lattice. The cohomology groups are 
\begin{displaymath}
\begin{array}{l}
H^1(G_{6.11}^{p=0}/\Gamma_{2\pi})=H^1(\mathfrak{\tilde{g}}_{6.11}^{p=0})=\langle \alpha^2,\,\alpha^3,\,\alpha^6\rangle \\ [-8pt]\\
H^2(G_{6.11}^{p=0}/\Gamma_{2\pi})=H^2(\mathfrak{\tilde{g}}_{6.11}^{p=0})=\langle \alpha^{23},\,\alpha^{26},\,\alpha^{36} \rangle \\[-8pt]\\
H^3(G_{6.11}^{p=0}/\Gamma_{2\pi})=H^3(\mathfrak{\tilde{g}}_{6.11}^{p=0})=\langle \alpha^{145},\,\alpha^{236}\rangle
\end{array}\\[-8pt]
\end{displaymath}

\smallskip

\noindent \emph{The subgroups $\Gamma_{2\pi/k}$ ($k \in \ZZ$) are also lattices
 if and only if $2\cos\left( \frac{2\pi}{k} \right) \in \mathbb{Z}$.} In all these cases we have
\begin{displaymath}
\begin{array}{l}
H^1(G_{6.11}^{p=0}/\Gamma_{\bar t})=\langle \alpha^6\rangle \\ [-8pt]\\
H^2(G_{6.11}^{p=0}/\Gamma_{\bar t})=\langle \alpha^{23}\rangle \\[-8pt]\\
H^3(G_{6.11}^{p=0}/\Gamma_{\bar t})=\langle \alpha^{145},\,\alpha^{236}\rangle \, .
\end{array} \\[-8pt]
\end{displaymath}

\medskip

\noindent $\bullet$ {\large{{$G_{6.12}^{-4p,p}$}}: 
\emph{ one can show that there is no lattice for $t=2\pi$}.

\medskip

\noindent $\bullet$ {\large{{$G_{5.13}^{-1-2q,q,r}\times\mathbb{R}$}}: 
 we must consider two different cases: if $r \in \mathbb{R}\smallsetminus \mathbb{Q}$, $\mathcal{A}(\Ad(\Gamma_{2\pi}))$ is connected, instead if $r=\frac{r_1}{r_2} \in \mathbb{Q}$, $\mathcal{A}(\Ad(\Gamma_{2\pi r_2}))$ is connected, but \emph{one can show that there is no lattice for both these values of $t$}.
 
 \medskip
 
\noindent $\bullet$ {\large{{$G_{5.14}^{0}\times\mathbb{R}$}}: \emph{for $\bar t=2\pi$ $\mathcal{A}(\Ad(\Gamma_{\bar t}))$ is connected and $\Gamma_{\bar t}$ is a lattice, then the only non-zero bracket of the Lie algebra $\tilde{\mathfrak{g}}\cong \mathfrak{g}_{3.1}\oplus\mathbb{R}^3$ is $[X_2,X_5]=X_1$} and the cohomology groups are
\begin{displaymath}
\begin{array}{rl}
H^1(G_{5.14}^{0}\times\mathbb{R}/\Gamma_{2\pi})=H^1(\mathfrak{\tilde{g}})=&\langle\alpha^2,\,\alpha^3,\,\alpha^4,\,\alpha^5,\,\alpha^6\rangle\, , \\ [-8pt]\\
H^2(G_{5.14}^{0}\times\mathbb{R}/\Gamma_{2\pi})=H^2(\mathfrak{\tilde{g}})=&\langle  \alpha^{12},\,\alpha^{15},\,\alpha^{23},\,\alpha^{24},\,\alpha^{26},\,\alpha^{34},\,\alpha^{35},\,\alpha^{36},\,\alpha^{45},\,\alpha^{46},\,\alpha^{56}\rangle\, , \\[-8pt]\\
H^3(G_{5.14}^{0}\times\mathbb{R}/\Gamma_{2\pi})=H^3(\mathfrak{\tilde{g}})=&\langle \alpha^{123},\,\alpha^{124},\,\alpha^{125},\,\alpha^{126},\,\alpha^{135},\,\alpha^{145},\,\alpha^{156},\,\alpha^{234},\,\alpha^{236},\\
&\alpha^{246}, \,\alpha^{345},\,\alpha^{346},\,\alpha^{356},\,\alpha^{456}\rangle \, .
\end{array}\\[-8pt]
\end{displaymath}
\emph{The subgroups $\Gamma_{2\pi/k}$ ($k \in \ZZ$) are again lattices 
if and only if $2\cos\left( \frac{2\pi}{k} \right) \in \mathbb{Z}$}, in particular for all these values we have the same invariants and the cohomology groups are:
\begin{displaymath}
\begin{array}{l}
H^1(G_{5.14}^{0}\times\mathbb{R}/\Gamma_{\frac{2\pi}{k}})=\langle\alpha^2,\,\alpha^5,\,\alpha^6\rangle\, , \\ [-8pt]\\
H^2(G_{5.14}^{0}\times\mathbb{R}/\Gamma_{\frac{2\pi}{k}})=\langle  \alpha^{12},\,\alpha^{15},\,\alpha^{26},\,\alpha^{34},\,\alpha^{56}\rangle\, , \\[-8pt]\\
H^3(G_{5.14}^{0}\times\mathbb{R}/\Gamma_{\frac{2\pi}{k}})=\langle \alpha^{125},\,\alpha^{126},\,\alpha^{156},\,\alpha^{234},\,\alpha^{345},\,\alpha^{346}\rangle\, .
\end{array}
\end{displaymath}

We note that these groups are isomorphic to the cohomology groups of the Lie algebra $\mathfrak{g}_{5.14}^{0}\oplus \RR$.

\medskip

\noindent $\bullet$ {\large{{$G_{5.17}^{p,-p,r}\times\mathbb{R}$}}:  
again we must consider two different cases: \emph{if $r \in \mathbb{R}\smallsetminus \mathbb{Q}$, $\mathcal{A}(\Ad(\Gamma_{2\pi}))$ is connected, but we have no lattice, instead if $r=\frac{r_1}{r_2} \in \mathbb{Q}$, $\mathcal{A}(\Ad(\Gamma_{2\pi r_2}))$ is connected and $\Gamma_{2\pi r_2}$ is a lattice if and only if $e^{2\pi pr_2}+e^{-2\pi pr_2}=h \in \mathbb{Z}$.}

So for these values of $p$ and $r_2$ the Lie algebra $\tilde{\mathfrak{g}}$ is $\mathbb{R}^6$ for $p=0$, while for $p\neq 0$ the non zero brackets in $\tilde \cg$ are given by $$[X_1,X_5]=pX_1,\;[X_2,X_5]=pX_2,\;[X_3,X_5]=-pX_3,\;[X_4,X_5]=-pX_4.$$ 
Thus if $p\neq 0$, $\tilde{\mathfrak{g}}$ is isomorphic to $\mathfrak{g}_{5.7}^{1,-1,-1}\oplus \RR$. The cohomology groups of the solvmanifold $G_{5.17}^{p,-p,r}\times\mathbb{R}/\Gamma_{2\pi}$ are 
\begin{displaymath}
\begin{array}{l}
H^1(G_{5.17}^{p,-p,r}\times\mathbb{R}/\Gamma_{2\pi r_2})=H^1(\mathfrak{\tilde{g}})=\langle \alpha^5,\,\alpha^6\rangle\, , \\ [-8pt]\\
H^2(G_{5.17}^{p,-p,r}\times\mathbb{R}/\Gamma_{2\pi r_2})=H^2(\mathfrak{\tilde{g}})=\langle  \alpha^{13},\,\alpha^{14},\,\alpha^{23},\,\alpha^{24},\,\alpha^{56}\rangle\, , \\[-8pt]\\
H^3(G_{5.17}^{p,-p,r}\times\mathbb{R}/\Gamma_{2\pi r_2})=H^3(\mathfrak{\tilde{g}})=\langle \alpha^{135},\,\alpha^{136},\,\alpha^{145},\,\alpha^{146},\,\alpha^{235},\,\alpha^{236},\,\alpha^{245},\,\alpha^{246}\rangle\, .
\end{array}
\end{displaymath}

To study other lattices we consider the case $r \in \mathbb{Z}$ and then $t=\frac{2\pi}{k}$: the characteristic polynomial of $\exp(t \ad_{X_5})$ has coefficients that depends strongly on the relation between $k$ and $r$, so it is difficult to determine in general for which values of $k$ they are integer.

For this reason we consider only particular values of $k$:

\noindent (a) $k=2$: if $r$ is even we have a lattice if and only if $h-2=n^2$ for some $n \in \mathbb{Z}$ and the cohomology groups of the solvmanifold are:

\begin{description}[leftmargin=2cm,
style=nextline]
\itemsep3pt \parskip3pt \parsep0pt 
\item[{\rm if $p \neq 0$}]
$H^1(G_{5.17}^{p,-p,r}\times\mathbb{R}/\Gamma_{\pi})=\langle \alpha^5, \alpha^6 \rangle, \\ [-8pt]\\H^2(G_{5.17}^{p,-p,r}\times\mathbb{R}/\Gamma_{\pi})=\langle \alpha^{56} \rangle, \\ [-8pt]\\H^3(G_{5.17}^{p,-p,r}\times\mathbb{R}/\Gamma_{\pi})=\{0\};$
\item[{\rm if $p=0$}]\smallskip
$H^1(G_{5.17}^{p,-p,r}\times\mathbb{R}/\Gamma_{\pi})=\langle \alpha^3, \alpha^4, \alpha^5, \alpha^6 \rangle, \\ [-8pt]\\
H^2(G_{5.17}^{p,-p,r}\times\mathbb{R}/\Gamma_{\pi})=\langle \alpha^{12}, \alpha^{34}, \alpha^{35}, \alpha^{36}, \alpha^{45}, \alpha^{46}, \alpha^{56} \rangle, \\[-8pt]\\
H^3(G_{5.17}^{p,-p,r}\times\mathbb{R}/\Gamma_{\pi})=\langle \alpha^{123}, \alpha^{124}, \alpha^{125}, \alpha^{126}, \alpha^{345}, \alpha^{346}, \alpha^{356}, \alpha^{456} \rangle .$
\end{description}
We note that for $p\neq 0$ these groups are isomorphic to the cohomology groups of the Lie algebra.

\smallskip

If $r$ is odd we have a lattice if there exists an integer $n$ such that $h+2=n^2$ and

\begin{description}[leftmargin=2.5cm,
style=nextline]
\itemsep3pt \parskip3pt \parsep0pt 
\item[{\rm if $p \neq 0$}]
$H^1(G_{5.17}^{p,-p,r}\times\mathbb{R}/\Gamma_{\pi})=\langle \alpha^5, \alpha^6 \rangle, \\ [-8pt]\\
H^2(G_{5.17}^{p,-p,r}\times\mathbb{R}/\Gamma_{\pi})=\langle \alpha^{13}, \alpha^{14}, \alpha^{23}, \alpha^{24}, \alpha^{56} \rangle, \\[-8pt]\\
 H^3(G_{5.17}^{p,-p,r}\times\mathbb{R}/\Gamma_{\pi})=\langle \alpha^{135}, \alpha^{136}, \alpha^{145}, \alpha^{146}, \alpha^{235}, \alpha^{236}, \alpha^{245}, \alpha^{246} \rangle;$
 \smallskip
\item[{\rm if $p=0$}]
$H^1(G_{5.17}^{p,-p,r}\times\mathbb{R}/\Gamma_{\pi})=\langle \alpha^5, \alpha^6 \rangle, \\ [-8pt]\\
H^2(G_{5.17}^{p,-p,r}\times\mathbb{R}/\Gamma_{\pi})=\langle \alpha^{12}, \alpha^{13}, \alpha^{14}, \alpha^{23}, \alpha^{24}, \alpha^{34},  \alpha^{56} \rangle, \\ [-8pt]\\
{}\!\!\!\!\!\!\!\!\begin{array}{lll} &H^3(G_{5.17}^{p,-p,r}\times\mathbb{R}/\Gamma_{\pi})=&\langle \alpha^{125}, \alpha^{126}, \alpha^{135}, \alpha^{136}, \alpha^{145}, \alpha^{146}, \alpha^{235}, \alpha^{236}, \\ && \alpha^{245}, \alpha^{246}, \alpha^{345}, \alpha^{346} \rangle . \end{array}$ 
\end{description}

\medskip

\noindent (b) $k=4$: if $r\equiv 0 \mod 4$ then the characteristic polynomial has integer coefficients if and only if $p=0$ and for this value our matrix is integer, so there is the lattice.

\smallskip

\noindent $H^1(G_{5.17}^{p,-p,r}\times\mathbb{R}/\Gamma_{\frac{\pi}{2}})=\langle \alpha^3, \alpha^4, \alpha^5, \alpha^6 \rangle, \\ [-8pt]\\
H^2(G_{5.17}^{p,-p,r}\times\mathbb{R}/\Gamma_{\frac{\pi}{2}})=\langle \alpha^{12}, \alpha^{34}, \alpha^{35}, \alpha^{36}, \alpha^{45}, \alpha^{46}, \alpha^{56} \rangle, \\[-8pt]\\
H^3(G_{5.17}^{p,-p,r}\times\mathbb{R}/\Gamma_{\frac{\pi}{2}})=\langle \alpha^{123}, \alpha^{124}, \alpha^{125}, \alpha^{126}, \alpha^{345}, \alpha^{346}, \alpha^{356}, \alpha^{456} \rangle .$

\smallskip

If $r\equiv 1 \mod 4$ again we have a lattice only if $h+2=n^2$ for some $n \in \mathbb{Z}$ and

\begin{description}[leftmargin=2cm,
style=nextline]
\itemsep3pt \parskip3pt \parsep0pt 
\item[{\rm if $p \neq 0$}]
$H^1(G_{5.17}^{p,-p,r}\times\mathbb{R}/\Gamma_{\frac{\pi}{2}})=\langle \alpha^5, \alpha^6 \rangle, \\[-8pt]\\
 H^2(G_{5.17}^{p,-p,r}\times\mathbb{R}/\Gamma_{\frac{\pi}{2}})=\langle \alpha^{13}+\alpha^{24},\,\alpha^{14}-\alpha^{23},\,\alpha^{56} \rangle, \\[-8pt]\\
  H^3(G_{5.17}^{p,-p,r}\times\mathbb{R}/\Gamma_{\frac{\pi}{2}})=\langle \alpha^{135}+\alpha^{245},\,\alpha^{145}-\alpha^{235},\,\alpha^{146}-\alpha^{236},\,\alpha^{136}+\alpha^{246}\rangle ;$\smallskip
\item[{\rm if $p=0$}]
$H^1(G_{5.17}^{p,-p,r}\times\mathbb{R}/\Gamma_{\frac{\pi}{2}})=
\langle \alpha^5, \alpha^6 \rangle, \\ [-8pt]\\
H^2(G_{5.17}^{p,-p,r}\times\mathbb{R}/\Gamma_{\frac{\pi}{2}})=\langle \alpha^{12},\alpha^{13}+\alpha^{24},\,\alpha^{14}-\alpha^{23},\, \alpha^{34},  \alpha^{56} \rangle, \\
[-8pt]\\
 {}\!\!\!\begin{array}{ll} H^3(G_{5.17}^{p,-p,r}\times\mathbb{R}/\Gamma_{\frac{\pi}{2}})= &\langle \alpha^{125}, \alpha^{126}, \alpha^{135}+\alpha^{245},\,\alpha^{145}-\alpha^{235},\,\alpha^{146}-\alpha^{236},\,\\&\alpha^{136}+\alpha^{246}, \alpha^{345}, \alpha^{346} \rangle. \end{array}$
\end{description}
For $r=1$ they are isomorphic to the cohomology groups of the Lie algebra.
\smallskip

If $r\equiv 2 \mod 4$ then again there is a lattice only if $p=0$ and we have an isomorphism with the invariant cohomology groups:
$$
\begin{array}{ll}
H^1(G_{5.17}^{p,-p,r}\times\mathbb{R}/\Gamma_{\frac{\pi}{2}})=\langle \alpha^5, \alpha^6 \rangle, \\[-8pt]\\
 H^2(G_{5.17}^{p,-p,r}\times\mathbb{R}/\Gamma_{\frac{\pi}{2}})=\langle \alpha^{12}, \alpha^{34}, \alpha^{56} \rangle, \\[-8pt]\\
  H^3(G_{5.17}^{p,-p,r}\times\mathbb{R}/\Gamma_{\frac{\pi}{2}})=\langle \alpha^{125}, \alpha^{126}, \alpha^{345}, \alpha^{346} \rangle .
  \end{array}
  $$
If $r\equiv 3 \mod 4$ we get same coefficients as if $r\equiv 1 \mod 4$ and then we have a lattice only if $h+2=n^2$ for some $n \in \mathbb{Z}$.

\begin{description}[leftmargin=2cm,
style=nextline]
\itemsep3pt \parskip3pt \parsep0pt 
\item[{\rm  if $p \neq 0$}]
$H^1(G_{5.17}^{p,-p,r}\times\mathbb{R}/\Gamma_{\frac{\pi}{2}})=\langle \alpha^5, \alpha^6 \rangle, \\[-8pt]\\
 H^2(G_{5.17}^{p,-p,r}\times\mathbb{R}/\Gamma_{\frac{\pi}{2}})=\langle \alpha^{14}+\alpha^{23},\,\alpha^{13}-\alpha^{24},\,\alpha^{56} \rangle, \\ [-8pt]\\
 H^3(G_{5.17}^{p,-p,r}\times\mathbb{R}/\Gamma_{\frac{\pi}{2}})=\langle \alpha^{145}+\alpha^{235},\alpha^{135}-\alpha^{245},\,\alpha^{136}-\alpha^{246},\,\,\alpha^{146}+\alpha^{236}\rangle$;\smallskip
\item[{\rm  if $p=0$}]
$H^1(G_{5.17}^{p,-p,r}\times\mathbb{R}/\Gamma_{\frac{\pi}{2}})=\langle \alpha^5, \alpha^6 \rangle, \\[-8pt]\\
H^2(G_{5.17}^{p,-p,r}\times\mathbb{R}/\Gamma_{\frac{\pi}{2}})=\langle \alpha^{12},\alpha^{14}+\alpha^{23},\,\alpha^{13}-\alpha^{24},\, \alpha^{34},  \alpha^{56} \rangle, \\ [-8pt]\\
{}\!\!{} \begin{array}{ll} H^3(G_{5.17}^{p,-p,r}\times\mathbb{R}/\Gamma_{\frac{\pi}{2}})=&\langle \alpha^{125}, \alpha^{126},\alpha^{145}+\alpha^{235},\alpha^{135}-\alpha^{245},\,\\
& \alpha^{136}-\alpha^{246},\,\alpha^{146}+\alpha^{236},\alpha^{345}, \alpha^{346} \rangle .
\end{array}$
\end{description}
Again we have the isomorphism with the invariant cohomology, but only for $r=-1$.

\medskip

\noindent $\bullet$ {\large{{$G_{5.18}^{0}\times\mathbb{R}$}}: for $t=2\pi$ $\mathcal{A}(\Ad(\Gamma_t))$ is connected, \emph{there is a lattice and $\tilde{\mathfrak{g}}\cong\mathfrak{g}_{5.1}\oplus\mathbb{R}$}, so
\begin{displaymath}
\begin{array}{ll}
H^1(G_{5.18}^{0}\times\mathbb{R}/\Gamma_{2\pi})=H^1(\mathfrak{\tilde{g}})=&\langle \alpha^3,\,\alpha^4,\,\alpha^5,\,\alpha^6\rangle \\ [-8pt]\\
H^2(G_{5.18}^{0}\times\mathbb{R}/\Gamma_{2\pi})=H^2(\mathfrak{\tilde{g}})=&\langle  \alpha^{13},\,\alpha^{15},\,\alpha^{14}+\alpha^{23},\,\alpha^{24},\,\alpha^{25},\,\alpha^{34},\,\alpha^{36},\,\alpha^{46},\,\alpha^{56}\rangle \\ [-8pt]\\
H^3(G_{5.18}^{0}\times\mathbb{R}/\Gamma_{2\pi})=H^3(\mathfrak{\tilde{g}})=&\langle \alpha^{123},\,\alpha^{125},\,\alpha^{134},\,\alpha^{135},\,\alpha^{136},\,\alpha^{146}+\alpha^{236},\,\alpha^{156},\\ 
&\alpha^{234},\, \alpha^{235},\,\alpha^{245}, \,\alpha^{246},\,\alpha^{256},\,\alpha^{346}\rangle
\end{array}
\end{displaymath}

\smallskip

\noindent Again we can have other lattices $\Gamma_{2\pi/k}$ only for $k=2,3,4,6$ and

\begin{description}[leftmargin=2.3cm,
style=nextline]
\itemsep3pt \parskip3pt \parsep0pt 
\item[{\rm $k=2$}]
$H^1(G_{5.18}^{0}\times\mathbb{R}/\Gamma_{\pi})=\langle \alpha^5,\,\alpha^6\rangle \\ [-8pt]\\
H^2(G_{5.18}^{0}\times\mathbb{R}/\Gamma_{\pi})=\langle \alpha^{13},\,\alpha^{14}+\alpha^{23},\,\alpha^{24},\,\alpha^{34},\,\alpha^{56}\rangle \\[-8pt]\\
H^3(G_{5.18}^{0}\times\mathbb{R}/\Gamma_{\pi})=\langle \alpha^{125},\,\alpha^{135},\,\alpha^{136},\,\alpha^{146}+\alpha^{236},\,\alpha^{235},\,
\alpha^{245}, \,\alpha^{246},\,\alpha^{346}\rangle$
\item[{\rm $k=3,4,6$}]
$H^1(G_{5.18}^{0}\times\mathbb{R}/\Gamma_{\frac{2\pi}{k}})=\langle \alpha^5,\,\alpha^6\rangle \\ [-8pt]\\
H^2(G_{5.18}^{0}\times\mathbb{R}/\Gamma_{\frac{2\pi}{k}})=\langle \alpha^{13}+\alpha^{24},\,\alpha^{34},\,\alpha^{56}\rangle \\[-8pt]\\
H^3(G_{5.18}^{0}\times\mathbb{R}/\Gamma_{\frac{2\pi}{k}})=\langle \alpha^{125},\,\alpha^{135}+\alpha^{245},\,\alpha^{136}+\alpha^{246},\,\alpha^{346}\rangle$
\end{description}
The last case is isomorphic to the cohomology of the Lie algebra.

\medskip

\noindent $\bullet$ {\large{{$G_{4.6}^{-2p,p}\times\mathbb{R}^2$}}: \emph{for $t=2\pi$ $\mathcal{A}(\Ad(\Gamma_t))$ is connected, but there is not a lattice}.

\medskip

\noindent $\bullet$ {\large{{$G_{3.5}^{0}\times\mathbb{R}^3$}}: \emph{for $t=2\pi$ $\mathcal{A}(\Ad(\Gamma_t))$ is obviously connected and we have a lattice, in particular $\tilde \cg_{3.5}^{0}\times\mathbb{R}^3 \cong\RR^6$, so $G_{3.5}^{0}\times\mathbb{R}^3/\Gamma_{2\pi}$ is diffeomorphic to a 6-torus}.

Again, {the subgroups $\Gamma_{2\pi/k}$ ($k \in \ZZ$) are again lattices 
if and only if $2\cos\left( \frac{2\pi}{k} \right) \in \mathbb{Z}$}, in particular for all these values  the cohomology groups are always isomorphic to the invariant ones:
$$
\begin{array}{l}
H^1(G_{3.5}^{0}\times\mathbb{R}^3/\Gamma_{\frac{2\pi}{k}})=\langle \alpha^3, \alpha^4, \alpha^5, \alpha^6 \rangle, \\[-8pt]\\
 H^2(G_{3.5}^{0}\times\mathbb{R}^3/\Gamma_{\frac{2\pi}{k}})=\langle \alpha^{12}, \alpha^{34}, \alpha^{35}, \alpha^{36}, \alpha^{45}, \alpha^{46}, \alpha^{56} \rangle, \\[-8pt]\\
 H^3(G_{3.5}^{0}\times\mathbb{R}^3/\Gamma_{\frac{2\pi}{k}})=\langle \alpha^{123}, \alpha^{124}, \alpha^{125}, \alpha^{126}, \alpha^{345}, \alpha^{346}, \alpha^{356}, \alpha^{456} \rangle .
 \end{array}
 $$
We list the Lie algebras $\cg$ and the ones obtained by deformation $\tilde \cg$ in Table 2.

 \begin{table}[h!]\label{liealg}
\caption{Deformed Lie algebras}
\begin{center}
\setlength{\extrarowheight}{6pt}
\begin{tabular}{ | c |  c | }
 \hline
$\cg$  & $\tilde \cg$\\ [-15pt]
{} & {} \\
\hline \hline
  $\cg_{6.8}^{p=0}$ & $\cg_{4.5}\oplus \RR^2$ \\[-15pt]
{} & {} \\ \hline
  $\cg_{6.10}^{a=0}$ & $\cg_{4.1}\oplus \RR^2$\\[-15pt]
{} & {} \\ \hline
$\cg_{6.11}^{p=0}$ & $\cg_{4.6}^{-2k, k}\oplus \RR^2$\\[-15pt]
{} & {} \\ \hline
$\cg_{5.14}^{0}\oplus \mathbb{R}$ & $\mathfrak{g}_{3.1}\oplus\mathbb{R}^3$\\[-15pt]
{} & {} \\ \hline
$\cg_{5.17}^{p,-p,r}\oplus \mathbb{R}$ & $\begin{array}{cr}
\RR^6, &p =0\\[-15pt]
{} & {} \\
\mathfrak{g}_{5.7}^{1,-1,-1}\oplus \RR, & p \neq 0\\[-15pt]
{} & {} \end{array}$\\
\hline
$\cg_{5.18}^{0}\oplus\mathbb{R}$ & $\mathfrak{g}_{5.1}\oplus\mathbb{R}$\\ [-15pt]
{} & {} \\ \hline
$\cg_{3.5}^{0}\oplus \mathbb{R}^3$ & $\RR^6$ \\ [-15pt]
{} & {} \\
\hline
\end{tabular}
\end{center}
\label{default}
\end{table}%

\section{Symplectic structures and Lefschetz properties}\label{sympl}

Let us study symplectic structures on the solvmanifolds we consider. 

\smallskip

In general, since $\tilde{G}/\Gamma$ is diffeomorphic to $G/\Gamma$, (Theorem \ref{diffeo}), symplectic structures on the modified Lie algebra $\tilde{\mathfrak{g}}$   yield not $G$-invariant symplectic structures $\tilde \omega$ on $G/\Gamma$ (where $\Gamma$ is the lattice from which ${\mathcal A} (\Ad_G (\Gamma))$ is connected).
Recall from the previous Sections that these manifolds $G/\Gamma$ cover solvmanifolds $G/\bar \Gamma$.
Observe that in general the symplectic forms $\tilde \omega$ are defined only on the covering $G/\Gamma$ and (in general) not on the covered manifolds $G/\bar \Gamma$.

\smallskip

Let us start with the indecomposable case.
We know from the classification of symplectic structures on six dimensional solvable Lie algebras (see \cite{Macri}) that only the solvmanifolds $G_{6.10}^{a=0}/\Gamma_t$ have invariant symplectic structures (inherited by $\mathfrak{g}_{6.10}^{a=0}$). The generic invariant symplectic form is 
\begin{equation}\label{6.10} \omega=\omega_{1,6}\alpha^{16}+\omega_{2,3}\alpha^{23}+\omega_{2,6}\alpha^{26}+\omega_{3,6}\alpha^{36}+\omega_{4,5}\alpha^{45}+\omega_{4,6}\alpha^{46}+\omega_{5,6}\alpha^{56} \end{equation}
 with $\omega_{1,6}\omega_{2,3}\omega_{4,5} \neq 0$ ($\det \tilde \omega \neq 0$).

Next, let us look for non invariant symplectic structures.

In the case of $G_{6.8}^{p=0}$ we have that also $\tilde{\mathfrak{g}}_{6.8}^{p=0}$ does not admit symplectic structures. 

\smallskip

As for $G_{6.10}^{a=0}$, symplectic forms on the modified Lie algebra $\tilde{\mathfrak{g}}_{6.10}^{a=0}$ (which is isomorphic to $\cg_{4.1} \oplus \RR^2$) yield  
the (in general non $G_{6.10}^{a=0}$-invariant) symplectic form on $G_{6.10}^{a=0}/\Gamma_{2\pi}$
$$
\tilde{\omega}=\omega+\eta\, ,
$$
where $\omega$ is given by \eqref{6.10} and $\eta={\omega_{3,4}\alpha^{34}+\omega_{3,5}\alpha^{35}}$ is the ``new part'' (recall that the invariant cohomology, i.e. the cohomology of the Lie algebra $\cg_{6.10}^{a=0}$, is contained in the de Rham cohomology of $G_{6.10}^{a=0}/\Gamma_{2\pi}$, cf. Remark~\ref{incl}).

\smallskip

Again $\tilde{\mathfrak{g}}_{6.11}^{p=0}$ does not admit symplectic structures.

\smallskip

Let us consider now the decomposable case. 

The Lie algebra $\mathfrak{g}=\mathfrak{g}_{5.14}^{0}\oplus\mathbb{R}$ admits the symplectic form $$\omega =\omega_{1,2}\alpha^{12}+\omega_{1,5}\alpha^{15}+\omega_{2,5}\alpha^{25}+\omega_{2,6}\alpha^{26}+\omega_{3,4}\alpha^{34}+\omega_{3,5}\alpha^{35}+\omega_{4,5}\alpha^{45}+\omega_{5,6}\alpha^{56}\, , $$ with $\det (\omega_{i,j})\neq 0$,  
$$\tilde{\omega}=\omega + \omega_{2,3}\alpha^{23}+\omega_{2,4}\alpha^{24}+{\omega_{3,6}\alpha^{36}}+{\omega_{4,6}\alpha^{46}}\; \; {\text{with }} \det (\tilde{\omega}_{i,j})\neq 0\, .
$$
The Lie algebra $\mathfrak{g}=\mathfrak{g}_{5.17}^{p,-p,r}\oplus\mathbb{R}$ admits symplectic structures only for particular values of the parameters $p$ and $r$, \cite{Macri}, but

\noindent If $p=0$:
$\tilde{\mathfrak{g}} $ is isomorphic to $\mathbb{R}^6$ so it is symplectic.

\noindent If $p\neq 0$:
	 $\tilde{\mathfrak{g}} $ has generic symplectic form
$$\tilde{\omega}=\omega_{1,3}\alpha^{13}+\omega_{1,4}\alpha^{14}+\omega_{1,5}\alpha^{15}+\omega_{2,3}\alpha^{23}+\omega_{2,4}\alpha^{24}+\omega_{2,5}\alpha^{25}+\omega_{3,5}\alpha^{35}+\omega_{4,5}\alpha^{45}+\omega_{5,6}\alpha^{56}$$ with $\det (\tilde{\omega}_{i,j})\neq 0$.  

\smallskip

The Lie algebra $\mathfrak{g}=\mathfrak{g}_{5.18}^{0}\oplus\mathbb{R}$ admits the symplectic form
$$\omega =\omega_{1,3}(\alpha^{13}+\alpha^{24})+\omega_{1,5}\alpha^{15}+\omega_{2,5}\alpha^{25}+\omega_{3,4}\alpha^{34}+\omega_{3,5}\alpha^{35}+\omega_{4,5}\alpha^{45}+\omega_{5,6}\alpha^{56}\, ,$$ with $\omega_{1,3} \omega_{5,6} \neq 0$, but the solvmanifold $G_{5.18}^{0}\times\mathbb{R}/\Gamma_{2\pi}$ has also a not invariant symplectic structure inherited by $\tilde{\mathfrak{g}} $:
$$\tilde{\omega}=\omega+{\omega_{1,3}\alpha^{13}+\omega_{1,4}(\alpha^{14}+\alpha^{23})}+{\omega_{3,6}\alpha^{36}}+{\omega_{4,6}\alpha^{46}}\, ,
$$ with $\det (\tilde{\omega}_{i,j})\neq 0$.
 
 \smallskip
 
The Lie algebra $\mathfrak{g}=\mathfrak{g}_{3.5}^{0}\oplus\mathbb{R}^3$ admits symplectic structures, but $\tilde{\mathfrak{g}} $ is isomorphic to $\mathbb{R}^6$ so the solvmanifold $G_{3.5}^{0}\times\mathbb{R}^3/\Gamma_{2\pi}$ admits obviously also a not invariant symplectic structure.

\begin{definition}\label{half}
An $SU(3)$ structure  six-dimensional manifold $M$ (i.e., an $SU(3)$ reduction of the
frame bundle of $M$) defines a non-degenerate
2-form $\omega$, an almost-complex structure $J$, and a complex volume form $\Psi$. The SU(3) structure
is called \emph{half-flat} if $\omega \wedge \Omega$ and the real part of $\Psi$ are closed \cite{CS}.
If in addition $\omega$ is closed, the half-flat structure is called \emph{symplectic}.
\end{definition}

\begin{proof}[Proof of Proposition~\ref{not-inv}]
We use the classification in \cite{FMOU} together with the above discussion on symplectic forms on the solvmanifolds, possibly coming from forms on the modified Lie algebra (cf. Table 2).

By \cite[Proposition 4.2]{FMOU}, there is no $4\oplus 2$ decomposable Lie algebra admitting symplectic half-flat structures. Hence the symplectic forms on $G_{6.10}^{a=0}/\Gamma_{2\pi}$ we found are not  half-flat (recall that the modified Lie algebra $\tilde{\mathfrak{g}}_{6.10}^{a=0}$  is isomorphic to $\cg_{4.1} \oplus \RR^2$). 

By \cite[Proposition 4.3]{FMOU}, the $5 \oplus 1$ decomposable Lie algebras having symplectic half-flat structures 
are $\mathfrak{g}_{5.1}\oplus \mathbb{R}$ (which is isomorphic to $\ch_3$ in the notation of \cite{FMOU}),  $\mathfrak{g}_{5.17}^{p,-p,r}\oplus \mathbb{R}$  for $p\geq 0$ and $r=1$ and $\mathfrak{g}_{5.7}^{p,q,r}\oplus \mathbb{R}$ for $p=q=-1$ and $r=1$. 
\end{proof}

Next we consider the hard Lefschetz property.  Recall that for a symplectic manifold $(M^{2n},\omega)$ the \textit{hard Lefschetz property} holds if for every $0\leq k \leq n$ the homomorphism 
\begin{displaymath}
\begin{array}{cccc}
L^{n-k}: & H^k(M)& \rightarrow & H^{2n-k}(M) \\
& \left[ \alpha \right]  & \mapsto & [\omega^{n-k}\wedge \alpha]
\end{array}
\end{displaymath}
 is an isomorphism.
 
More in general, $(M^{2n},\omega)$ is called $s$-\textit{Lefschetz}  if $L^{n-k}$ 
is an isomorphism for all $k \leq s$ \cite{FM1}. The property of being $0$-Lefschetz is equivalent to \emph{cohomologically symplectic}, i.e., there exists $\omega\in H^{2}(M)$ such that $\omega^{n}\not=0$.

We need to consider $G_{6.10}^{a=0}/\Gamma_{\bar t}, \,G_{5.14}^{0}\times \mathbb{R}/\Gamma_{\bar t},\,G_{5.18}^{0}\times \mathbb{R}/\Gamma_{\bar t},\,G_{5.17}^{p,-p,r}\times \mathbb{R}/\Gamma_{\bar t},\,G_{3.5}^{0}\times \mathbb{R}^3/\Gamma_{\bar t}$ (for the above choices of $\bar t$). 

As a start, we consider the generic (non invariant) symplectic form $\tilde{\omega}$ on $G/\Gamma$ with ${\mathcal A} (\Ad_G (\Gamma))$ connected.

\begin{prop}\label{HL}
The hard Lefschetz property does not hold for the symplectic form $\tilde{\omega}$ on $G_{6.10}^{a=0}/\Gamma_{2\pi}, \,G_{5.14}^{0}\times \mathbb{R}/\Gamma_{2\pi},\,G_{5.18}^{0}\times \mathbb{R}/\Gamma_{2\pi}$.
More in general, these solvmanifolds are 0-Lefschetz but not $1$- and $2$-Lefschetz.
\\The hard Lefschetz property holds for the symplectic form $\tilde{\omega}$ on $G_{5.17}^{p,-p,r}\times \mathbb{R}/\Gamma_{2\pi r_2} \,  (r=\frac{r_1}{r_2}\in \Q) ,\,G_{3.5}^{0}\times \mathbb{R}^3/\Gamma_{2\pi}$.
\end{prop}

\begin{proof}
We describe the proof of the first part only for the first solvmanifold because the other cases are quite similar and the proposition comes from direct computation of the morphisms $L^{n-k}$.
$G_{6.10}^{a=0}/\Gamma_{2\pi}$ is $0$-Lefschetz because it is cohomologically symplectic. 
By a direct computation, we find that for every $\alpha \in \Lambda^2\Omega(G/\Gamma)$, $\alpha^{1235}$ never appears in $\tilde{\omega} \wedge \alpha$, but it is a generator of $H^4(G/\Gamma)$, so $L^1:H^2 (M) \rightarrow H^4 (M)$ cannot be an isomorphism. This implies that $G_{6.10}^{a=0}/\Gamma_{2\pi}$ is not $2$-Lefschetz.
Moreover, $\tilde \omega^2 \wedge \alpha^3$ is cohomologous to zero, so $L^2:H^1 (M) \rightarrow H^5 (M)$ cannot be an isomorphism (recall that $\alpha^3$ is a generator of $H^1(G_{6.10}^{0}/\Gamma_{2\pi})$) and $G_{6.10}^{a=0}/\Gamma_{2\pi}$ is not $1$-Lefschetz.\\
The second part of the proposition is quite obvious because the Lie algebra $\tilde{\mathfrak{g}}$ is in both cases isomorphic to $\mathbb{R}^6$.
\end{proof}

If we consider the invariant symplectic form $\omega$, then one can see that the hard Lefschetz property holds for $(G_{5.17}^{0,0,r}\times \mathbb{R}/\Gamma_{\bar t},\omega)$ and $(G_{5.17}^{p,-p,\pm 1}\times \mathbb{R}/\Gamma_{\bar t},\omega)$ and it does not hold for $(G_{5.14}^{0}\times \mathbb{R}/\Gamma_{\bar t},\omega)$, $(G_{5.18}^{0}\times \mathbb{R}/\Gamma_{\bar t},\omega)$ \cite[Proposition 7.12]{bock} and for  $(G_{6.10}^{a=0}/\Gamma_{\bar t}, \omega)$ \cite[Proposition 7.9]{bock}. 

\section{Minimal Models and formality}\label{minimal}

We now compute the minimal model of the solvmanifolds we found. We use a method developed by Oprea and Tralle \cite{OT, OT2} that exploits the Mostow fibration.

\begin{theorem} \cite{OT, OT2}
Let $F \rightarrow E \rightarrow B$ be a fibration and let $U$ be the largest $\pi_1(B)$-submodule of $H^*(F,\mathbb{Q})$ on which $\pi_1(B)$ acts nilpotently. Suppose that $H^*(F)$ is a vector space of finite type and that $B$ is a nilpotent space, then in the Sullivan model of the fibration 

\begin{displaymath}
\xymatrix{\mathcal{A}(B) \ar[r]  & \mathcal{A}(E) \ar[r] & \mathcal{A}(F) \\
(\Lambda X, d_X) \ar[r] \ar[u] & (\Lambda (X\oplus Y), D)  \ar[u] \ar[r] & (\Lambda Y, d_Y) \ar[u]^\alpha }
\end{displaymath}
the differential graded algebra homomorphism $\alpha: (\Lambda Y, d_Y)  \to \mathcal{A}(F) $ induces an isomorphism $\alpha^*: H^*(\Lambda Y, d_Y)\rightarrow U$.
\end{theorem}

In particular in the case of the Mostow fibration
$$
N / \Gamma_N = (N \Gamma) / \Gamma \hookrightarrow G / \Gamma
\longrightarrow G / (N \Gamma) = \T^k\, ,
$$
we can construct the minimal model $(\Lambda (X\oplus Y), D)$ of the solvmanifold using the models of of the base $\T^k$ (for almost abelian solvmanifolds $k=1$, i.e., we have a circle $S^1$) and the fibre $N / \Gamma_N$ (actually of its submodule $U$).

In general, finding $U$ is very difficult, but when the solvmanifold is almost nilpotent (in particular almost abelian), the monodromy action of $\ZZ \cong \pi_1(S^1)$ on $H^*(N / \Gamma_N)$ is exploited by the (transpose of) twist action that defines the semi-direct sum $\mathfrak{g}=\mathbb{R}\ltimes \mathfrak{n}$, that in our case is just $\exp (t \, \ad_{X_6})$ (see \cite[Theorems 3.7 and 3.8]{OT}). Unfortunately, with this method, in some of our examples  we cannot find  the model uniquely, because we can have different choices for the construction of $(\Lambda (X\oplus Y), D)$. However, we can identify the right one, knowing the cohomology groups from the previous computations.

We write down the computations explicitly only for some of the solvmanifolds, trying to show all the possible different cases, and for the others we only give the minimal model.

\medskip

\noindent $\bullet$ {\large{$G_{6.8}^{p=0}/\Gamma_{2\pi}$}}:
\begin{displaymath}
U= \left\{ \begin{array}{l}
\langle \alpha^4, \alpha^5 \rangle \subset H^1(\mathfrak{n}) \\[-8pt]\\
\langle \alpha^{45} \rangle \subset H^2(\mathfrak{n}) \\[-8pt]\\
\langle \alpha^{123} \rangle \subset H^3(\mathfrak{n}) \\[-8pt]\\
\langle \alpha^{1234}, \alpha^{1235} \rangle \subset H^4(\mathfrak{n}) \\[-8pt]\\
\langle \alpha^{12345} \rangle = H^5(\mathfrak{n})
\end{array} \right. \, ,
\end{displaymath}
and  a minimal model for $U$ is $\mathcal{M}_U=(\Lambda (x_1,y_1,z_3),0)$. 

The minimal model of the base $S^1$ is $(\Lambda(A_1),0)$ and the minimal model of the solvmanifold is $\mathcal{M}= (\Lambda (A_1,x_1,y_1,z_3),0)$.

\smallskip

\noindent $\bullet$ {\large{$G_{6.8}^{p=0}/\Gamma_{\pi,\frac{\pi}{2},\frac{\pi}{3}}$}}:
$$\mathcal{M}= (\Lambda (A_1,x_2,\beta _3,y_3),D), \qquad DA=Dx=Dy=0, \; D\beta=x^2\, .
$$

\smallskip

\noindent $\bullet$ {\large{$G_{6.10}^{a=0}/\Gamma_{2\pi}$}}:
\begin{displaymath}
\begin{array}{lcr}
U=H^*(\mathfrak{n})  & \Rightarrow & \mathcal{M}_U= (\Lambda (x_1,y_1,z_1,p_1,q_1),0)\\
\end{array}
\end{displaymath}
The minimal model of the solvmanifold is $\mathcal{M}= (\Lambda (A_1,x_1,y_1,z_1,p_1,q_1),D)$, but we have 7 different choices for $D$:
\begin{enumerate}
\item $D\equiv 0$
\item $DA=Dx=Dy=Dz=Dp=0, \; Dq=Ax$
\item $DA=Dx=Dy=Dz=0 \; Dp=Ax, \; Dq=Ay$
\item $DA=Dx=Dy=Dz=0 \; Dp=Ax, \; Dq=Ap$
\item $DA=Dx=Dy=0 \; Dz=Ax \; Dp=Ay, \; Dq=Az$
\item $DA=Dx=Dy=0 \; Dz=Ax \; Dp=Az, \; Dq=Ap$
\item $DA=Dx=0 \; Dy=Ax \; Dz=Ay \; Dp=Az, \; Dq=Ap$
\end{enumerate}
Computing the cohomology groups of these commutative differential graded algebras (c.d.g.a.) and comparing with those of $G_{6.10}^{a=0}/\Gamma_{2\pi}$, we find that (4) is the right one.

\smallskip

\noindent $\bullet$ {\large{$G_{6.10}^{a=0}/\Gamma_{\pi,\frac{\pi}{2},\frac{\pi}{3}}$}}:
\begin{displaymath}
U= \left\{ \begin{array}{lcr}
\langle \alpha^1, \alpha^2, \alpha^3 \rangle \subset H^1(\mathfrak{n})& & \\[-8pt]\\
\langle \alpha^{12}, \alpha^{13}, \alpha^{23}, \alpha^{45} \rangle \subset H^2(\mathfrak{n}) & & \\[-8pt]\\
\langle \alpha^{123}, \alpha^{145}, \alpha^{245}, \alpha^{345} \rangle \subset H^3(\mathfrak{n}) & \Rightarrow & \mathcal{M}_U= (\Lambda (x_1,y_1,z_1,t_2,\beta_3),d), \\[-8pt]\\
\langle \alpha^{1245}, \alpha^{1345}, \alpha^{2345} \rangle \subset H^4(\mathfrak{n}) & & dx=dy=dz=dt=0, \; d\beta= t^2\\[-8pt]\\
\langle \alpha^{12345} \rangle = H^5(\mathfrak{n}) & &
\end{array} \right.
\end{displaymath}
The minimal model of the solvmanifold is $\mathcal{M}= (\Lambda (A_1,x_1,y_1,z_1,t_2,\beta_3),D)$, but we have 13 different choices for $D$. Fortunately, only the following are not isomorphic with each other:

\begin{enumerate}
\item $DA=Dx=Dy=Dz=Dt=0, \; D\beta=t^2$
\item $DA=Dx=Dy=0, \; Dz=Ay \; Dt=0, \; D\beta=t^2$
\item $DA=Dx=0, \; Dy=Ax, \; Dz=Ay \; Dt=0, \; D\beta=t^2$
\end{enumerate}
Computing the cohomology groups of these c.d.g.a. and comparing with those of $G_{6.10}^{a=0}/\Gamma_{\pi,\frac{\pi}{2},\frac{\pi}{3}}$, we find that (3) is the right one.

\smallskip

\noindent $\bullet$ {\large{$G_{6.11}^{p=0}/\Gamma_{2\pi}$}}: 
$\mathcal{M}= (\Lambda (A_1,x_1,y_1,z_3),0)$

\smallskip

\noindent $\bullet$ {\large{$G_{6.11}^{p=0}/\Gamma_{\pi,\frac{\pi}{2},\frac{\pi}{3}}$}}: 
$\mathcal{M}= (\Lambda (A_1,x_2,\beta _3,y_3),D), \; DA=Dx=Dy=0, \; D\beta=x^2$.

\begin{remark}
\textnormal{The model $(7)$ in the case of $G_{6.10}^{a=0}/\Gamma_{2\pi}$ has cohomology groups isomorphic to the cohomology groups of $G_{6.10}^{a=0}/\Gamma_{\pi,\frac{\pi}{2},\frac{\pi}{3}}$, and conversely the first model in $G_{6.10}^{a=0}/\Gamma_{\pi,\frac{\pi}{2},\frac{\pi}{3}}$ has  the same cohomology groups as the ones of $G_{6.10}^{a=0}/\Gamma_{2\pi}$.}
\end{remark}

\smallskip

\noindent $\bullet$ {\large $G_{5.14}^{0}\times\mathbb{R}/\Gamma_{2\pi}$}:
$$\mathcal{M}= (\Lambda (u_1,A_1,x_1,y_1,z_1,t_1),D), \; Du=DA=Dx=Dy=Dz=0,\,Dt=Ax\, .
$$

\noindent $\bullet$ {\large $G_{5.14}^{0}\times\mathbb{R}/\Gamma_{\frac{2\pi}{k}}$}:
$$\mathcal{M}= (\Lambda (u_1,A_1,x_1,y_1,z_2,\beta_3),D), \; Du=DA=Dx=Dz=0,\,Dy=Ax,D\beta=z^2
$$

\smallskip 

\noindent $\bullet$ {\large $G_{5.17}^{p,-p,r}\times\mathbb{R}/\Gamma_{2\pi r_2} \,  (r=\frac{r_1}{r_2}\in \Q)$}:

\noindent If  $p\neq 0$: $U= \left\{ \begin{array}{l}
\langle \alpha^{13}, \alpha^{14}, \alpha^{23}, \alpha^{24} \rangle \subset H^2(\mathfrak{n}) \\[-8pt]\\
\langle \alpha^{1234} \rangle = H^4(\mathfrak{n})
\end{array} \right. \,$. 

To list all the generators in $\mathcal{M}_U$ is almost impossible in this case: to every degree we need to add several generators to get the isomorphism in cohomology, but in this way we improve the number of generators needed.

With $\mathcal{M}^n$ we mean the subalgebra of $\mathcal{M}$ generated by all generators of $\mathcal{M}$ of degree $n$, then $\mathcal{M}_U^1=\{0\}$, $\mathcal{M}_U^2=(\Lambda(x_2,y_2,z_2,t_2),0)$ and $\forall n>2 \quad \mathcal{M}_U^n$ can be computed by induction (\cite[Theorem 2.24]{FOT}).
\\Then the minimal model of the solvmanifold is 

\[
\begin{array}{ll}
  \mathcal{M}=& (\Lambda (u_1,A_1,\mathcal{M}_U),D),\; Du=DA=Dx=Dy=Dz=Dt=0, \\&D|_{\mathcal{M}_U^n}\equiv d \; \forall n>2\, .
\end{array}
\]

\noindent If  $p= 0$: 
\begin{displaymath}
\begin{array}{lcr}
U=H^*(\mathfrak{n})  & \Rightarrow & \mathcal{M}_U= (\Lambda (x_1,y_1,z_1,t_1),0).\\
\end{array}
\end{displaymath}
The minimal model of the solvmanifold is $\mathcal{M}= (\Lambda (u_1,A_1,x_1,y_1,z_1,t_1),D)$, with different choices for $D$, but we know that it is isomorphic to $\mathbb{R}^6$, then $D\equiv 0$.

\smallskip

\noindent $\bullet$ {\large $G_{5.17}^{p,-p,r}\times\mathbb{R}/\Gamma_{\pi}$}:

$r$ even: 

if $p\neq 0$, then $\mathcal{M}= (\Lambda (u_1,A_1,x_4,\beta_7),D), \; Du=DA=Dx=0,\,D\beta=x^2$.

if $p=0$, then 
$$\mathcal{M}= (\Lambda (u_1,A_1,x_1,y_1,z_2,\beta_3),D), \;Du=DA=Dx=Dy=Dz=0,\,D\beta=z^2\,.
$$

 $r$ odd: 

if $p\neq 0$ we have the same model of $t=2\pi$.

If $p=0$ we are not able to list all the generators of $\mathcal{M}_U$, but as in the case of $t=2\pi r_2,\; p\neq 0$ we have 
\[
\begin{array}{ll}
  \mathcal{M}=& (\Lambda (u_1,A_1,\mathcal{M}_U),D), \; Du=DA=Dx=Dy=Dz=Dt=Ds=Dq=0 
  \\
  &D|_{\mathcal{M}_U^n}\equiv d \; \forall n>2\, .
\end{array}
\]

\smallskip 

\noindent $\bullet$ {\large $G_{5.17}^{p,-p,r}\times\mathbb{R}/\Gamma_{\frac{\pi}{2}}$}:

\noindent $r\equiv 0 \mod 4$: $p=0$ and we have the same computation of the case $t=\pi$ with $r$ even.\\
$r\equiv 1 \mod 4$: $\mathcal{M}= (\Lambda (u_1,A_1,\mathcal{M}_U),D), \; Du=DA=0,\,D|_{\mathcal{M}_U}\equiv d$.\\
$r\equiv 2 \mod 4$: $p=0$ and
$$\mathcal{M}= (\Lambda (u_1,A_1,x_2,y_2\beta_3,\gamma_3),D), \;Du=DA=Dx=Dy=0,D\beta=x^2, D\gamma=y^2 \, .
$$
$r\equiv 3 \mod 4$:  the model is the same of the case $r\equiv_4 1$.

\smallskip

\noindent $\bullet$ {\large $G_{5.18}^{0}\times\mathbb{R}/\Gamma_{2\pi}$}:
$$\mathcal{M}= (\Lambda (u_1,A_1,x_1,y_1,z_1,t_1),D), \; Du=DA=Dx=Dy=0,\,Dz=Ax,\, Dt=Ay\, .$$

\smallskip 

\noindent $\bullet$ {\large $G_{5.18}^{0}\times\mathbb{R}/\Gamma_{\pi}$}:
\[
\begin{array}{ll}
  \mathcal{M}=& (\Lambda (u_1,A_1,\mathcal{M}_U),D), \quad Du=DA=Dx=Dy=Dz=Dt=0,\, Ds=Ax,\,\\ &Dq=Ay, D|_{\mathcal{M}_U^n}\equiv d \; \forall n>2\, .
\end{array}
\]

\smallskip 

\noindent $\bullet$ {\large $G_{5.18}^{0}\times\mathbb{R}/\Gamma_{\frac{\pi}{3},\frac{\pi}{2},\frac{2\pi}{3}}$}:
\[
\begin{array}{ll}
  \mathcal{M}=&(\Lambda (u_1,A_1,\mathcal{M}_U),D), \quad Du=DA=Dx=Dy=0,\, Dz=Ax,\,Dt=Ay,
  \\
  &D|_{\mathcal{M}_U^n}\equiv d \; \forall n>2\, .
\end{array}
\]

\smallskip

\noindent $\bullet$ {\large $G_{3.5}^{0}\times\mathbb{R}^3/\Gamma_{2\pi}$}: 
$ \mathcal{M}=(\Lambda(w_1,v_1,A_1,x_1,y_1),0)$.

\smallskip 

\noindent $\bullet$ {\large $G_{3.5}^{0}\times\mathbb{R}^3/\Gamma_{\frac{2\pi}{k}}$}:
$$\mathcal{M}= (\Lambda (w_1,v_1,u_1,A_1,x_2,\beta_3),D), \;Dw=Dv=Du=DA=Dx=0,\,D\beta=x^2\, .
$$

\medskip 

Next, we use these models to decide which of these solvmanifolds are formal.

\smallskip

Recall that a manifold is \textit{formal} if so is its minimal model, that is there exists a c.g.d.a. homomorphism $\psi: \mathcal{M}\rightarrow H^*(\mathcal{M})$ that induces the identity in cohomology. In particular the definition implies that every closed generator must be sent to its cohomology class, while the others must be sent to zero, but not always this construction gives the identity also in higher dimension of cohomology.

\begin{proof}[Proof of Theorem \ref{formality}] If we have the explicit computation of the model of the solvmanifold we can define $\psi$ and see directly if the map induced in cohomology is the identity or not.\\
For example in the case $G_{6.10}^{a=0}/\Gamma_{2\pi}$, following the construction we have 
$$
\psi:  A  \mapsto  [A], \, x \mapsto [x], y \mapsto [y], z \mapsto [z], p \mapsto 0, q \mapsto 0\, , 
$$
then $\psi([Ap])=\psi([0])=0$, but $[Ap]\neq 0$, so $\psi^*$ is not the identity.

A similar proof can be done for the other solvmanifolds of which we have the explicit minimal model.

If we do not have the model of the solvmanifold we consider a weaker notion of formality introduced by M. Fern\'andez and V. Mu\~noz in \cite{FM1}, namely $s$-formality. 
It can be stated as follows (see \cite[Lemma 2.7]{FM2}). Let $M$ be a manifold with minimal model $\mathcal{M}=(\bigwedge V, d)$. Then $M$ is $s$-formal if and
only if there is a map of differential algebras
$\phi: (\bigwedge V^{\leq s}, d)\to (H^*
(M), d = 0)$, such that the map $\phi^\ast: H^*(\bigwedge V^{\leq s}, d)\to H^*(M)$ induced on cohomology is equal to the map $i^\ast: H^*(\bigwedge V^{\leq s}, d) \to (H^*(\bigwedge V, d)=H^*(M)$ induced by the inclusion $i: (\bigwedge V^{\leq s}, d) \to (\bigwedge V, d)$.\\
To study formality of these solvmanifolds we use the following theorem

\begin{theorem}\cite{FM1}
Let $M$ be a connected and orientable compact differentiable manifold of dimension $2n$, or $2n-1$. Then $M$ is formal if and only if is $(n-1)$-formal.
\end{theorem}

We can apply this theorem to the c.d.g.a. $\mathcal{M}_U$ because the manifold $M$ in the hypothesis can be replaced by a real c.d.g.a.  $\mathcal{A}$ with the following  properties:
\begin{itemize}
\item $H^0(\mathcal{A})=\mathbb{R}$
\item $\forall i>\dim(\mathcal{A})\quad  H^i(\mathcal{A})=0$
\item $H^{\dim(M)-i}(\mathcal{A})\cong H^i(\mathcal{A})$ (Poincar\'e duality)
\end{itemize}

$\mathcal{M}_U$ has dimension $4$ and it has always these three characteristics, so to prove that it is formal we must only prove that it is $1$-formal. In particular in the cases in which we do not have explicit model $\mathcal{M}_U$ is always simply connected because $U^1=\{0\}$, so it is $1$-formal and then the theorem states that it is always formal.

Now we use formality of $(\mathcal{M}_U,d_{\mathcal{M}_U})$ to study formality of the model of the solvmanifold $(\mathcal{M},D)$: if $\mathcal{M}$ has differential $D$ such that $D|_{\mathcal{M}_U}\equiv d_{\mathcal{M}_U}$, then it is obviously formal, otherwise we can show that it is not formal defining the map $\psi$ in a similar way to the case $G_{6.10}^{a=0}/\Gamma_{2\pi}$.
\end{proof}

In particular one can verify that all the not formal solvmanifolds that we have considered are 0-formal but not 1-formal.

\medskip

\noindent \textbf{
Acknowledgments.}  We would like to thank Anna Fino and Luis Ugarte for many
useful comments and suggestions.

\end{document}